\documentclass[preprint,12pt]{elsarticle}

\usepackage{amssymb}
\usepackage{amsmath}
\usepackage{amsthm}
\usepackage{array}
\usepackage{rotating}
\usepackage{vmargin}
\setmarginsrb{2.2cm}{1cm}{2.2cm}{3cm}{1cm}{1cm}{1.5cm}{1.5cm}
\newtheorem{thm}{Theorem}
\newtheorem{cor}{Corollary}
\newtheorem{prop}{Proposition}
\newtheorem{conj}{Conjecture}
\newtheorem{lem}[thm]{Lemma}
\newdefinition{rem}{Remark}
\newdefinition{defi}{Definition}
\newproof{pf}{Proof}
\newcolumntype{M}[1]{>{\raggedright}m{#1}}

\journal{*****}

\begin{document}

\begin{frontmatter}

\author{Thierry COMBOT\fnref{label2}}
\ead{thierry.combot@u-bourgogne.fr}
\address{IMB, Universit\'e de Bourgogne, 9 avenue Alain Savary, 21078 Dijon Cedex }

\title{Symbolic integration on planar differential foliations}

\author{}

\address{}

\begin{abstract}
We consider the problem of symbolic integration of $\int G(x,y(x)) dx$ where $G$ is rational and $y(x)$ is a non algebraic solution of a differential equation $y'(x)=F(x,y(x))$ with $F$ rational. As $y$ is transcendental, the Galois action generates a family of parametrized integrals $I(x,h)=\int G(x,y(x,h)) dx$. We prove that $I(x,h)$ is either differentially transcendental or up to parametrization change satisfies a linear differential equation in $h$ with constant coefficients, called a telescoper. This notion generalizes elementary integration. We present an algorithm to compute such telescoper given a priori bound on their order and degree $\hbox{ord},N$ with complexity $\tilde{O}(N^{\omega+1} \hbox{ord}^{\omega-1}+N\hbox{ord}^{\omega+3})$. For the specific foliation $y=\ln x$, a more complete algorithm without an a priori bound is presented. Oppositely, non existence of telescoper is proven for a classical planar Hamiltonian system. As an application, we present an algorithm which always finds, if they exist, the Liouvillian solutions of a planar rational vector field, given a bound large enough for some notion of complexity height.
\end{abstract}
\end{frontmatter}

\section{Introduction}

Let us consider a non autonomous differential equation
\begin{equation}\label{eq1}
\frac{d}{dx} y(x)=F(x,y(x))
\end{equation}
where $F\in\mathbb{K}(x,y)$ is a rational fraction with coefficients in a field $\mathbb{K}$, a finite extension of $\mathbb{Q}$. A first integral of equation \eqref{eq1} is a function $\mathcal{F}$ (possibly defined only locally in a neighbourhood of a point $(x_0,y_0)$) such that
$$\frac{d}{dx} \mathcal{F}(x,y(x))= \partial_x\mathcal{F}(x,y(x))+F(x,y(x))\partial_y\mathcal{F}(x,y(x))=0.$$
A general solution of equation \eqref{eq1} is a function $y(x,h)$, analytic at least on a polydisc $\mathcal{D}$ in $x,h$ centred on $(x_0,h_0)$ such that for all $(x,h)\in\mathcal{D}$
$$\partial_x y(x,h)=F(x,y(x,h)),\; \partial_h y(x,h)\neq 0 \hbox{ identically}$$
Such a general solution can have differential properties with respect to $h$ also. In particular following \cite{Casale}, and \cite{cheze2020symbolic}, if there exists a general solution not differentially transcendental in $h$, then equation \eqref{eq1} admits a non constant first integral of the following $4$ types
\begin{itemize}
\item A rational first integral $\mathcal{F}\in\mathbb{C}(x,y)$
\item A $k$-Darbouxian first integral, $(\partial_y \mathcal{F})^k=R \in\mathbb{C}(x,y)$ for some $k\in\mathbb{N}^*$
\item A Liouvillian first integral, $\partial_y^2 \mathcal{F}/\partial_y \mathcal{F}=R \in \mathbb{C}(x,y)$
\item A Ricatti first integral, $\mathcal{F}=\mathcal{F}_1/\mathcal{F}_2$ where $\mathcal{F}_1,\mathcal{F}_2$ are a $\mathbb{C}(x)$ basis of a differential equation of the form $\partial_y^2 \mathcal{F}/ \mathcal{F}=R \in \mathbb{C}(x,y)$.
\end{itemize}
As proved in \cite{cheze2020symbolic}, it appears that in each case, the first integral can be defined (up to addition, affine and homographic transformation) by the single rational function $R$, which can in fact be chosen in $\mathbb{K}(x,y)$. Such first integrals can up found in polynomial time given a bound on the degree of $R$. For Darbouxian and Liouvillian first integrals, we can define \emph{the integrating factor} which is $U(x,y)=\partial_y \mathcal{F}$. It is algebraic in the Darbouxian case, and hyperexponential in the Liouvillian case.

The subject of this article is the symbolic computation of an integral of the form
$$\int G(x,y(x)) dx ,\quad G\in\mathbb{K}(x,y)$$
where $y(x)$ is a solution of equation \eqref{eq1}. Let us first remark that if a rational first integral exists, then $y(x)$ is algebraic and thus $G(x,y(x))$ is algebraic. If equation \eqref{eq1} does not admit a rational first integral, then for some exceptional initial conditions, equation \eqref{eq1} could have an algebraic solution $y(x)$. These situations will not be the main scope of the article, and we will focus on the case where $y(x)$ is transcendental. Thus from now on, we will assume that \textbf{equation \eqref{eq1} does not admit a rational first integral}. When $y(x)$ is transcendental, using Galois action, the function $y(x)$ can be replaced by $y(x,h)$ a general solution of equation \eqref{eq1}. The integration problem becomes  the computation of
$$I(x,h)=\int G(x,y(x,h)) dx.$$

We want to find differential polynomial relations on the function $I(x,h)$ with coefficients in $\mathcal{O}(h)$, holomorphic functions in $h$.

\begin{thm}\label{thm1}
If $I(x,h)$ satisfies a non constant differential equation in $\mathcal{O}(h)[x,y,\partial_h y,\partial_{hh} y,\dots,$ $I,\partial_h I,\partial_h^2 I,\dots]$, then up to parametrization change in $h$ it satisfies a differential equation of the form $L I =(\partial_h y)^{\hbox{ord}(L)} H$ where $H\in\mathbb{C}(x,y)$ and
\begin{itemize}
\item $L\in\mathbb{C}[\partial_h^k]\partial_h^j,\; j\in \{0,\dots, k-1\}$ when equation \eqref{eq1} has a $k$-Darbouxian first integral.
\item $L=\partial_h^j, j\in \mathbb{N}$ when equation \eqref{eq1} has a Liouvillian first integral.
\item $L=\partial_h^j,\; j\in \{0,1\}$ when equation \eqref{eq1} has a Ricatti first integral or $y$ differentially transcendental.
\end{itemize}
\end{thm}

Such an equation will be called a telescoper, and the right part the certificate. This generalizes what happens for integrals of algebraic functions depending on a parameter $h$, which always admit a telescoper with rational coefficients in $h$ \cite{lipshitz1988diagonal}. These can be algorithmically computed \cite{chyzak2014abc}, and as they always exist, the algorithms terminate. In the transcendental case, such equations do not always exist, and if they do, it is delicate to bound a priori their order and degree (see section $4$). However, as we will see, such telescoper can be found up to a given bound on their order and degree of $H$ in polynomial time.

Remark that the existence of a Ricatti first integral does not allow more complicated telescopers than in the differentially transcendental case. As the algebraic case is excluded by assumption, the only special foliation cases are those with a Liouvillian first integral and a Darbouxian first integral. Since in these cases an integrating factor $U$ exists and it defines the first integral, we will only have to take into account the existence of such integrating factor. If $U$ is a $k$-th root of a rational fraction, then we are in the $k$-Darbouxian case, if it is hyperexponential not algebraic we are in the Liouvillian case, and if it does not exist then we will look only at telescopers of the form $L=\partial_h^j,\; j\in \{0,1\}$ (which leads in fact to elementary integrals). There could be several possible integrating factors, and we will then assume that the one given is the simplest one (i.e. an algebraic one if there exists one). Assuming the non existence of a rational first integral, we present in section $3$ an algorithm inspired by \cite{combot2021reduction} which computes an integrating factor up to a bound $N$ and always returns the simplest one.

\begin{thm}\label{thm2}
The algorithm \underline{\sf FindTelescoper} takes in input a rational function $G\in\mathbb{K}(x,y)$, a differential equation \eqref{eq1}, bounds $\hbox{ord},N$ on the order and degree, an initial point $(x_0,y_0)$ and optionally an integrating factor $U$, and returns either a telescoper, or ``None'', or FAIL.
The FAIL output can only happen when $(x_0,y_0)$ belongs to an algebraic set of codimension $1$ or if equation \eqref{eq1} admits a rational first integral. If ``None'' is returned then no telescoper of the form of Theorem \ref{thm1} with the given $U$ exists with order and degree lower or equal to the bounds. The algorithm runs in time $\tilde{O}(N^{\omega+1}\hbox{ord}^{\omega+3})$.
\end{thm}

The algorithm is probabilistic. If it returns a solution, it is always correct (even if the transcendence condition on equation \eqref{eq1} is not satisfied), but it can fail. The FAIL output can occur when the series expansion of integral $I$ with initial conditions $y(x_0)=y_0$ is abnormally close to an algebraic series. This corresponds to zeros of some extatic curve. The FAIL can also occur in a reduction step which will produce a Darboux polynomial for equation \eqref{eq1}. If the algorithm is launched with a differential equation \eqref{eq1} admitting a rational first integral, any initial condition $(x_0,y_0)$ will lie on a Darboux polynomial, and thus the failure output can give candidates for such rational first integral. If the algorithm is given an integrating factor which is not the simplest possible, i.e hyperexponential instead of algebraic, then a possible telescoper using the algebraic one could be missed.

If a telescoper is found, this implies that $I(x,h)$ is a $\partial$-finite function in $\mathbb{C}(x,y,\partial_h y)[\partial_x,\partial_h]$, as the PDE system $L I =(\partial_h y)^{\hbox{ord}(L)} H,\;\; \partial_x I=G$ has a finite dimensional space of solutions. Remark that the existence of such differential equation generalizes the case of elementary integration. Indeed, if $\int G(x,y(x,h)) dx$ is elementary, then
$$I(x,h)= F_0(x,y(x,h))+\sum\limits_{i=1}^{\ell} \lambda_i(x,y(x,h)) \ln F_i(x,y(x,h)).$$
Differentiating this expression in $x$, all the logs should disappear, and thus the $\lambda_i$ should be function of $h$ only. We can now apply repeatedly operators of the form $\partial_h-\frac{\partial_h \lambda_i}{\lambda_i}$ to $I$ to kill all the $\lambda_i$ on the right part. We thus obtain an operator $L\in \mathcal{O}(h)[\partial_h]$ such that $LI\in \mathcal{O}(h)[x,y(x,h),\partial_h y(x,h),\dots]$. Then by Theorem \ref{thm1}, the integral $I$ admits after parameter change a telescoper of the form  $\tilde{L} I =(\partial_h y)^{\hbox{ord}(\tilde{L})} H,\; \tilde{L}\in\mathbb{C}[\partial_h]$.

Conversely, the homogeneous part of a telescoper, the PDE $LI=0$, has constant coefficients and thus can be solved with exponential polynomial solutions in $h$, and then solutions can be recovered by variation of constants, but are in general non elementary. This implies however that an integral admitting a telescoper can be written $I(x,h)=J(x,y(x,h))$ with $J$ a Liouvillian function in $x,y$.

\section{Structure Theorem}

\begin{proof}[Proof of Theorem \ref{thm1}]
Consider a non trivial polynomial relation in $\mathcal{O}(h)[x,y,\partial_h y,\partial_h^2 y,\dots$ $I,\partial_h I,\partial_{hh} I,\dots]$ of order $p$. We will first prove that an affine relation exists. As 
$$\partial_x I(x,h)= G(x,y(x,h)),$$
we can write $I(x,h)$ as an integral of an element of the differential field (for derivation $\partial_x$)
$$\mathbb{L}=\mathcal{O}(h)(x,y(x,h),\partial_h y(x,h),\dots).$$
The Galois group of $\hbox{Gal}(\mathbb{L}(I(x,h)):\mathbb{L})$ is either additive or identity. If it is identity, then $I(x,h) \in \mathbb{L}$, and so we have a differential equation of order $0$. Else it is additive. Consider its action on $I(x,h)$ and its derivatives in $h$. There are two cases.
\begin{itemize}
\item There exists a $\mathcal{O}(h)$ linear combination of the integrals such that the Galois action on it is identity. Then this linear combination gives a linear differential equation in $h$ which applied to $I(x,h)$ gives an element of $\mathbb{L}$. Then there exists a differential equation of the form $L I=\tilde{H}$ with $L\in \mathcal{O}(h)[\partial_h]$ and $\tilde{H}\in\mathbb{L}$.
\item There are no linear combinations of the integrals such that the Galois action on it is identity. As the Galois action is additive, the action of one element should be of the form
$$\left(\int G(x,y(x,h)) dx,\partial_h \int G(x,y(x,h)) dx,\dots\right) \rightarrow $$
$$\left(\int G(x,y(x,h)) dx,\partial_h \int G(x,y(x,h)) dx,\dots\right)+v$$
with $v\in\mathbb{C}^{p+1}$. Thus the whole Galois group can be identified to translations given by vectors $v_1,v_2,\dots,v_m \in\mathbb{C}^{p+1}$. By hypothesis, these vectors form a generating family of $\mathbb{C}^{p+1}$ (as else there would exist a linear form vanishing on all the $v_i$). Thus we have by action of the Galois group on the invariant given by hypothesis
$$P\left(h,\left(\int G(x,y(x,h)) dx,\partial_h \int G(x,y(x,h)) dx,\dots\right)+\sum\limits_{i=1}^m \alpha_i v_i\right)=0 \;\; \forall \alpha\in\mathbb{C}^{m}$$
This implies that $P$ does not depend on $I,\partial_h I,\partial_{hh} I,\dots$, which is contradictory as $P$ is assumed to be non trivial.
\end{itemize}
We now know that we have an equation of the form 
\begin{equation}\label{eq2}
\sum a_i(h) \partial_h^i I=\tilde{H}
\end{equation}
with $a_i\in \mathcal{O}(h)$ and $\tilde{H}\in\mathbb{L}$. We can moreover assume that the order is minimal. Let us now look at the function $y(x,h)$. As the theorem is up to parametrization change in $h$, we can now assume that $\mathcal{F}(x,y(x,h))=h$ where $\mathcal{F}$ is the given first integral (if it exists).\\

\textbf{The $k$-Darboux case}. Equation \eqref{eq1} admits $\mathcal{F}$ a $k$-Darboux first integral. The differential Galois group in $x$ $\hbox{Gal}(\mathbb{C}(x,y(x,h)):\mathbb{C}(x))$ ($h$ seen as a parameter) acts on $\mathcal{F}$ by transformations of the form
\begin{equation}\label{eq3}
\mathcal{F} \rightarrow \xi \mathcal{F}+\alpha
\end{equation}
where $\xi$ is a $k$-th root of unity and $\alpha\in\mathbb{C}$. If the $k$ is chosen minimal and $y(x,h)$ not algebraic in $x$, then any transformation of the form \eqref{eq3} can be obtained by the action of the Galois group. This transformation acts on equation \eqref{eq2} giving after division by dominant coefficient
$$\partial_h^l I(x,h)+\sum\limits_{i=0}^{l-1} \frac{a_i(\xi h+\alpha)}{a_l(\xi h+\alpha)} \xi^{i-l} \partial_h^i I(x,h)=\frac{\tilde{H}(\xi h+\alpha,x,y(x,h),\xi \partial_h y(x,h),\dots)}{a_l(\xi h+\alpha)\xi^l}$$
Thus $I(x,h)$ should be simultaneous solutions of all these equations. Evaluating on two different $\alpha$ allows to find a relation of lower order (which contradicts the hypothesis) except if all these relations are the same. Thus $a_i(\xi h+\alpha)/a_l(\xi h+\alpha)$ are constant in $\alpha$, and so this relation has constant coefficients in $h$ and we can assume $a_l=1$
$$\partial_h^l I(x,h)+\sum\limits_{i=0}^{l-1} b_i \xi^{i-l} \partial_h^i I(x,h)=\tilde{H}(x,y(x,h),\xi \partial_h y(x,h),\dots)\xi^{-l}$$
Now choosing for $\xi$ several $k$-th root of unity should all give the same equation, thus $b_i\neq 0 \Rightarrow k\mid i-l$. Thus the equation can be rewritten
$$\sum\limits_{i=0}^{(l-j)/k} c_i \partial_h^{ki+j} I(x,h)=\tilde{H}(x,y(x,h),\xi \partial_h y(x,h),\dots)\xi^{-j}$$
As $(\partial_h y(x,h))^k$ can be expressed rationally in function of $y(x,h)$, we have
$$\sum\limits_{i=0}^{(l-j)/k} c_i \partial_h^{ki+j} I(x,h)=\sum\limits_{i=0}^{k-1} \tilde{H}_i(x,y(x,h)) (\partial_h y(x,h))^i \xi^{i-j} $$
As the right side should not depend on the choice of $\xi$ we deduce that the equation has the form
$$\sum\limits_{i=0}^{(l-j)/k} c_i \partial_h^{ki+j} I(x,h)=\tilde{H}_j(x,y(x,h)) (\partial_h y(x,h))^j= (\partial_h y(x,h))^l H(x,y(x,h)) $$
with $H\in\mathbb{C}(x,y(x,h))$.\\

\textbf{The Liouvillian case}. Equation \eqref{eq1} admits $\mathcal{F}$ a Liouvillian first integral, and no $k$-Darboux or rational first integral. The Galois action on $\mathcal{F}$ is now
\begin{equation}\label{eq4}
\mathcal{F} \rightarrow \xi \mathcal{F}+\alpha
\end{equation}
where $\xi\in \mathbb{C}^*$ and $\alpha\in\mathbb{C}$. As before, we obtain that the equation has constant coefficients 
$$\partial_h^l I(x,h)+\sum\limits_{i=0}^{l-1} b_i \xi^{i-l} \partial_h^i I(x,h)=\tilde{H}(x,y(x,h),\xi \partial_h y(x,h),\dots)\xi^{-l}$$
However this time $\xi$ can be arbitrary, and thus we need all $b_i=0$. The second derivative $\partial_h^2 y$ can be expressed rationally in function of $y,\partial_h y$ and thus we can simplify
$$\partial_h^l I(x,h)=\tilde{H}(x,y(x,h),\xi \partial_h y(x,h))\xi^{-l}$$
Now the right side should not depend on $\xi$ and thus should be homogeneous in $\partial_h y(x,h)$ of degree $l$, and thus
$$\partial_h^l I(x,h)=H(x,y(x,h))( \partial_h y(x,h))^l.$$
with $H\in\mathbb{C}(x,y(x,h))$.\\

\textbf{The Ricatti case}. Equation \eqref{eq1} admits $\mathcal{F}$ a Ricatti first integral, and no Liouvillian, $k$-Darboux or rational first integral. The Galois action on $\mathcal{F}$ is now
\begin{equation}\label{eq5}
\mathcal{F} \rightarrow \frac{a \mathcal{F}+b}{c\mathcal{F}+d}
\end{equation}
This includes affine transformations thus the previous reasoning applies giving an equation of the form
$$\partial_h^l I(x,h)=\tilde{H}(x,y(x,h),\partial_h y(x,h),\dots)$$
We can now consider the transformation $h\rightarrow 1/h$ and use the Faa di Bruno formula to compute
$$\partial_h^l I(x,1/h)=\frac{(-1)^l}{h^{2l}} \partial_2^l I\left(x,\frac{1}{h}\right)+\frac{l(l-1)(-1)^l}{h^{2l-1}} \partial_2^{l-1} I\left(x,\frac{1}{h}\right) +\dots $$
Thus except for $l=0,1$, the transformation applied on
$$\partial_h^l I(x,h)=\tilde{H}(x,y(x,h),\partial_h y(x,h),\dots)$$
will produce a differential equation with coefficients depending on $h$ even after dividing by the dominant coefficient. This would allow to reduce the order, which contradicts the hypothesis.
We can reduce $\partial_h^3 y$ in function of $y, \partial_h y, \partial_h^2 y$, thus the equation simplifies
$$\partial_h^l I(x,h)=\tilde{H}(x,y(x,h),\partial_h y(x,h),\partial_h^2 y(x,h))$$
For $l=0$, the right side should be invariant by transformations \eqref{eq5}. However, the lowest order non constant differential expression in $y$ invariant by \eqref{eq5} is the Schwarzian derivative, which is of order $3$. Thus $\tilde{H}$ is function of $x,y(x,h)$ only. For $l=1$, we divide both sides by $\partial_h y$
$$\frac{\partial_h I(x,h)}{\partial_h y}=\frac{\tilde{H}(x,y(x,h),\partial_h y(x,h),\partial_h^2 y(x,h))}{\partial_h y}.$$
The left side is invariant by transformation \eqref{eq5}, and thus so should be the right side. Thus $\tilde{H}/\partial_h y$ is function of $x,y$ only. So we can write in general
$$\partial_h^l I(x,h)=(\partial_h y)^l H(x,y(x,h)),\quad l\in\{0,1\}.$$

\textbf{Differentially transcendental case}. When $y(x,h)$ is differentially transcendental, the transformation group to consider are all analytic transformations $h\rightarrow \varphi(h)$, so containing in particular homographic transformations. Thus the same reasoning as in the Ricatti case apply, and thus the equation writes
$$\partial_h^l I(x,h)=\tilde{H}(x,y(x,h),\partial_h y(x,h),\dots),\;\; l=0,1$$
Now dividing both sides by $\partial_h^l y$, the left side is invariant by $\varphi$ and thus so should be the right side. As there are no non constant differential expressions in $y$ invariant by all the $\varphi$, we conclude that $\tilde{H}/(\partial_h y)^l=H$ is function of $x,y$ only.
\end{proof}

Let us remark that in the Ricatti and differentially transcendental case, having an equation of the form $\partial_h I(x,h)= \partial_h y H(x,y)$ allows to write
$$I(x,h)=J(x,y(x,h)) \hbox{ with } \partial_y J(x,y)=H(x,y),\; \partial_x J+F(x,y) \partial_y J=G(x,y)$$
and thus
$$I(x,h)= \int^{(x,y(x,h))} (G(x,y)-F(x,y)H(x,y))dx+H(x,y)dy$$
This is an integral of a rational closed $1$-form, and thus $I$ admits an elementary expression
$$I(x,h)=F_0(x,y(x,h))+\sum\limits_{i=1}^\ell \lambda_i \ln F_i(x,y(x,h))$$
Thus $I(x,h)$ is an elementary extension of $\mathbb{C}(x,y(x,h))$. Integrals with a telescoper and non elementary expressions can thus only appear in the $k$-Darbouxian and Liouvillian cases. In other words, non elementary integrals with a telescoper can only occur when equation \eqref{eq1} has an integrating factor.

\begin{cor}\label{cor1}
If $I(x,h)$ admits a telescoper and \eqref{eq1} admits an integrating factor $U$, then we can write $I(x,h)=J(x,y(x,h))$ where
\begin{equation}\label{eqliouv}
J(x,y)=\sum_{\lambda\in S} \sum_{r=0}^{m_{\lambda}} e^{\lambda \mathcal{F}} \mathcal{F}^r
\int e^{-\lambda \mathcal{F}} \sum_{i=1}^{m_{\lambda}} \sum_{j\in E_{\lambda,r,i}}\mathcal{F}^i U^j \omega_{\lambda,i,j,r}
\end{equation}
with $S\subset \mathbb{C}$ finite, $E_{\lambda,r,i} \subset \mathbb{Z}$ finite, $m_\lambda\in\mathbb{N}$, $\omega_{\lambda,i,j,r}$ rational $1$-forms.
\end{cor}

\begin{proof}
Let us first remark that
$$\partial_x J(x,y(x,h))= (\partial_x J+F(x,y)\partial_y J)(x,y(x,h))$$
Also, differentiating $\mathcal{F}(x,y(x,h))=h$ in $h$ gives $\partial_h y(x,h) U(x,y(x,h))=1$ and thus $\partial_h y(x,h) = U(x,y(x,h))^{-1}$. This allows to write
$$\partial_h J(x,y(x,h))= \partial_h y(x,h) \partial_y J(x,y(x,h))=(U(x,y)^{-1} \partial_y J)(x,y(x,h)).$$
Now this suggests to introduce the derivations
$$D_x=\partial_x+F(x,y)\partial_y,\quad D_h=U(x,y)^{-1} \partial_y$$
It appears that they commute
\begin{align*}
D_hD_x-D_xD_h & = & -D_x(U(x,y)^{-1})\partial_y+D_hF(x,y) \partial_y\\
 & = & \!\!\!\!\!\!\!\!\!\!\! ((\partial_x U(x,y)+F(x,y)\partial_y U(x,y))U(x,y)^{-2}+U(x,y)^{-1}\partial_y F(x,y)) \partial_y\\
  & = & U(x,y)^{-2}(\partial_x U(x,y)+F(x,y)\partial_y U(x,y)+U(x,y)\partial_y F(x,y)) \partial_y\\
  & = & U(x,y)^{-2}(\partial_x U(x,y)+\partial_y (FU(x,y))) \partial_y\\
  & = & U(x,y)^{-2}(\partial_x \partial_y\mathcal{F}(x,y)-\partial_y \partial_x\mathcal{F}(x,y)) \partial_y\\
  & = & 0
\end{align*}
using the formula $\partial_x\mathcal{F}=-F\partial_y \mathcal{F}=-FU$. So our telescoper and the given integrand $G$ for integral $I$ gives a PDE system for $J$
$$D_x J=G,\quad L J=U^{-l} H$$
where $L\in\mathbb{C}[D_h]$ is of order $l$. Noting $X=(J,D_hJ,\dots,D_h^{l-1}J)^\intercal$, we can write this as an integrable connection
$$D_x X= (G,D_hG,\dots,D_h^{l-1}G)^\intercal,\quad D_h X= MX+(0,\dots,0,U^{-l} H)^\intercal$$
where $M\in M_l(\mathbb{C})$ is the companion matrix of the characteristic polynomial of $L$. 
This is a non homogeneous linear system, and the fundamental matrix of solutions of the homogeneous part is
$$X=e^{M\mathcal{F}}$$
as $D_x \mathcal{F}=0$ and $D_h \mathcal{F}=1$. Thus we can search a particular solution of the form $e^{M\mathcal{F}} Y$ which gives
$$D_x(e^{M\mathcal{F}} Y)=e^{M\mathcal{F}} D_x Y=(G,D_hG,\dots,D_h^{l-1}G)^\intercal$$
$$D_h(e^{M\mathcal{F}} Y)=Me^{M\mathcal{F}} Y+e^{M\mathcal{F}} D_h Y= Me^{M\mathcal{F}} Y+(0,\dots,0,U^{-l} H)^\intercal$$
Thus the system becomes
$$ D_x Y= e^{-M\mathcal{F}} (G,D_hG,\dots,D_h^{l-1}G)^\intercal,\quad  D_h Y= e^{-M\mathcal{F}} (0,\dots,0,U^{-l} H)^\intercal $$
We can thus write $Y$ as an integral of a $1$ form whose coefficients are $\mathbb{C}(x,y)$-linear combination of exponential polynomials of $\mathcal{F}$ and powers of $U$. Then multiplying it with $e^{M\mathcal{F}}$, we obtain for $J$ an expression of the form
$$J(x,y)=\sum_{\lambda\in S} \sum_{r=0}^{v_{\lambda}} e^{\lambda \mathcal{F}} \mathcal{F}^r 
\int \sum_{\mu\in S} e^{-\mu \mathcal{F}} \sum_{i=1}^{m_{\lambda,\mu}} \sum_{j\in E_{\lambda,\mu,r,i}} \mathcal{F}^i U^j \omega_{\lambda,\mu,i,j,r}$$
where $\omega_{\lambda,\mu,i,j,r}$ are $1$-forms with rational coefficients, $E_{\lambda,\mu,r,i}\subset \mathbb{Z}$ finite and $S$ is the spectrum of $M$.

We will now act the Galois group (in $x$) on the exponentials $e^{\lambda\mathcal{F}}$. Remember that $J$ represent an integral, and thus is defined up to addition of a constant in $x$, and thus a multiplicative action on the expression of $J$ should be trivial. To let invariant the expression, the exponentials inside the integrals have to exactly compensate the exterior ones, which gives the expression
$$J(x,y)=\sum_{\lambda\in S} \sum_{r=0}^{m_{\lambda}} e^{\lambda \mathcal{F}} \mathcal{F}^r
\int e^{-\lambda \mathcal{F}} \sum_{i=1}^{m_{\lambda}} \sum_{j\in E_{\lambda,r,i}}\mathcal{F}^i U^j \omega_{\lambda,i,j,r}$$
\end{proof}

Remark that in the $k$-Darbouxian and Liouvillian case, there is a multiplicative Galois action
$$U\rightarrow \xi U,\; \mathcal{F} \rightarrow  \xi \mathcal{F}  $$
where $\xi$ should be a $k$-th root of unity in the $k$-Darbouxian case, and $\xi \in \mathbb{C}^*$ in the Liouvillian case. This puts additional restrictions on the possible formulas for $J$, in particular the spectrum $S$ has to be invariant by this multiplicative action. In the Liouvillian case, this restricts $S\subset \{0\}$, recovering the second point of Theorem \ref{thm1}.

\section{Algorithms}

\subsection{Rational first integrals}

We want to check the existence of rational first integrals for the equation $\partial_x y=F(x,y)$. Finding an algorithm for searching in general such rational first integral is an open problem, the Poincare problem \cite{Poincare}. Thus, here, as in \cite{BCCW},  we will only look for rational first integrals up to some bound $N$. In \cite{cheze2020symbolic}, an algorithm in $O(N^{1+\omega})$ is presented. However, this is not the one which is implemented. In fact, in current implementations, the slow point, contrary to asymptotic complexity estimates, is the computation of series solutions.\\

\noindent
\underline{SolSeries}\\
Input: A rational function $F$, an initial point $(x_0,y_0)$, the series order $\sigma$, an integer $d$\\
Output: A list of series equal to $(y(x)^i)_{i=0\dots d}$ up to order $\sigma$ where $y(x)$ is the solution with initial condition $(x_0,y_0)$\\

Theoretically, this algorithm could be implemented in $\tilde{O}(\sigma d)$ \cite{bostan2007fast}. In practice, writing such implementation would be long and the fast growing bit size of the coefficients decreases the efficiency, so we use a simpler approach in $O(\sigma^2)$, computing term by term the series $(y(x)^i)_{i=0\dots d}$. The critical part is to compute the convolution of sequences of coefficients of $y(x)^i$. An implementation trick is to use the formula
$$\alpha^{2i-1}=\tfrac{1}{2}(\alpha^{i-1}+\alpha^i)^2-\tfrac{1}{2}\alpha^{2i-2}-\tfrac{1}{2}\alpha^{2i}$$
to compute odd powers of $y(x)$ using only squares, which costs half because of the symmetry in the convolution. With these improvements, the computation of series is no longer the blocking point for the search of high degree first integrals. The extatic matrix, its kernel and the rational first integral are computed as in \cite{cheze2020symbolic}.\\

\noindent
\textbf{Example:} We consider the equation
$$\partial_x y(x)= \frac{ 17x^2y-17y^3-58xy-17y}{17x^3-17xy^2-58y^2-17x}$$
with the initial condition $(x_0,y_0)=(0,3)$
$$\begin{array}{|c|c|c|c|c|c|}\hline
\hbox{degree}          &  4  &  8  & 16  & 32  & 64  \\\hline
\hbox{time in s}       &0.01 &0.06 & 1.17& 39.2& 4551   \\\hline
\hbox{series time}     &0.01 &0.03 & 0.3 & 9.9 & 664   \\\hline
\end{array}$$

For $N=64$, $64$ series of order $2145$ are needed, showing that computing a few thousands terms is possible. The calculation time for the search of matrix kernel element is dominant, showing that the series computation is no more the limitation of this approach. For comparison, we write the table 6.5 of \cite{BCCW} done on the hypergeometric example
$$F=\frac{1-4n^2x^2y^2+4n^2y^2-4xyn^2}{4n^2(x-1)(x+1)}$$
$$\begin{array}{|c|c|c|c|c|c|}\hline
n               &  2  &  4   & 6    & 8  & 10   \\\hline
N               & 9   & 17   & 25   & 33 & 41   \\\hline
\hbox{time in s}&0.14 & 1.4  & 7.8  & 38 & 133\\\hline
\hbox{BCCW time}& 0.54& 12.54& 118.8& 592& 3247 \\\hline
\end{array}$$

\subsection{Search for a minimal integrating factor}

The possible structures for the telescoper depend on the existence of an integrating factor and its minimality. There is no algorithm to find with certainty an integrating factor, but using \cite{cheze2020symbolic}, we have an algorithm to find one up to a certain degree, defined by the maximum of degree of numerator and denominator of $\partial_y U/U$.

For reducing the integrating factor and checking its minimality, we will use the approach of \cite{combot2021reduction}. In general, the reduction is not guaranteed to be complete as an exceptional case can occur when trying to reduce Darbouxian first integrals to rational first integrals. However, here, using the assumption that no rational first integral exists, this not handled case cannot occur and moreover the algorithm will be simpler.\\

\noindent
\underline{IntegratingFactor}\\
Input: A rational function $F$, an initial point $(x_0,y_0)$, a bound $N$\\
Output: A minimal integrating factor $U$, ``None'', or FAIL.\\
\begin{enumerate}
\item If $(x_0,y_0)$ vanishes the denominator of $F$, return FAIL
\item Compute $Y=$\underline{Solseries}$(F,(x_0,y_0),\tfrac{3}{2}(N+1)(N+2),N)$.
\item Compute the lists of series at order $\tfrac{3}{2}(N+1)(N+2)$
$$L_1=Ye^{\int \partial_y F_{|y=Y_1} dx},\quad L_2=Y\int L_{1,1} \partial_y^2 F_{|y=Y_1} dx,\quad L_3=Y$$
\item Compute a non zero solution the system 
$$\sum_{i=1}^3 \sum_{j=0}^N L_j\hbox{coeff}_{y^j}P_i=O(x^{\tfrac{3}{2}(N+1)(N+2)})$$
where $P_i\in\mathbb{C}[x,y]$, $\deg P_i\leq N$. If none, return None.
\item if $P_2=0$ then if $P_1P_3=0$ or $\partial_x(P_1/P_3)+F \partial_y(P_1/P_3)+P_1/P_3 \partial_y F\neq 0$,
return FAIL else return $P_1/P_3$.
\item Else note
$$H=\partial_x (P_1/P_2)+F\partial_y(P_1/P_2)+P_1/P_2 \partial_y F+ \partial_y^2 F,\;
G=\partial_x(P_1/P_3)+F \partial_y(P_1/P_3)$$
\item if $H\neq 0$ then if $G=0$ or $\partial_x(H/G)+F\partial_y(H/G)+H/G\partial_y F\neq 0$ then return FAIL else return $H/G$.
\item Note $K_1=-\partial_y F-F P_1/P_2,\; K_2=P_1/P_2$. If $K_1dx+K_2dy$ has simple poles on $\mathbb{P}^2$ and residue set $S\subset \mathbb{Q}$, note $k$ their common denominator and return 
$$\left(\prod\limits_{\lambda\in S} \left(\hbox{gcd}(\hbox{num}(K_1dx+K_2dy)-\lambda d(\hbox{den}(K_1dx+K_2dy))  ,\hbox{den}(K_1dx+K_2dy))\right)^{k\lambda}\right)^{1/k}$$
where $d$ is the differential operator and the gcd holds on each component of the differential form.
\item Search for a rational solution $R$ of the system
$$\partial_y^2R+K_2 \partial_y R+(\partial_y K_2)R=0,\partial_x R+F \partial_y R+(K_1+FK_2)R=0$$
If one return $K_2+\partial_y R/R$, else return $\exp{\int K_1dx+K_2dy}$
\end{enumerate}

\begin{prop}\label{propirr}
Assume $\partial_x y=F(x,y)$ has no rational first integral. If it admits an integrating factor $U$ such that $\partial_y U/U$ has numerator and denominator degree $\leq N$, then algorithm \underline{IntegratingFactor} returns an integrating factor or FAIL. If moreover $U$ is algebraic, then the integrating factor returned is algebraic with minimal algebraic extension degree. If algorithm \underline{IntegratingFactor} returns ``None'', then no integrating factor $U$ such that $\partial_y U/U$ has numerator and denominator degree $\leq N$ exists. If algorithm \underline{IntegratingFactor} returns FAIL, then $(x_0,y_0)$ belongs to an algebraic set of codimension $1$.
\end{prop}

\begin{proof}
The steps 1-7 are exactly the step of algorithm LiouvillianFirstIntegrals of \cite{cheze2020symbolic}, with the only difference that FAIL is returned when a Darboux polynomial is detected. As we assumed $\partial_x y=F(x,y)$ has no rational first integral, only finitely many Darboux polynomials exist and thus form an algebraic set of codimension $1$, in which the algorithm is allowed to return FAIL. Thus if \eqref{eq1} admits an integrating factor $U$ such that $\partial_y U/U$ has numerator and denominator degree $\leq N$, the algorithm would find one in step $8$ as $U=\exp \int K_1dx+K_2 dy$. Thus, if ``None'' is returned, no such integrating factor exists. If an integrating factor is found in step $8$, it is guaranteed to be correct as the term $K_2=P_1/P_2$ satisfies the condition of step $7$, which is the condition for $K_2=P_1/P_2$ leading to a Liouvillian first integral.

We now need to verify if this integrating factor is minimal or look for a minimal one. We thus need to check if an algebraic integrating factor exists. Let us note $\mathcal{F}$ the Liouvillian first integral associated to the integrating factor we found. Assume equation \eqref{eq1} admits moreover a $k$-Darbouxian first integral $J$. We can then write
$$\mathcal{F}(x,y)=f(J(x,y)),\quad f \hbox{ one variable function}$$
We differentiate both sides giving
$$\partial_y \mathcal{F}(x,y)= \partial_y J(x,y) f'(J(x,y))$$
$$\partial_y^2 \mathcal{F}(x,y)= \partial_y^2 J(x,y) f'(J(x,y))+(\partial_y J(x,y))^2 f''(J(x,y))$$
and making the quotient and simplifying, we get
$$\frac{K_2(x,y)- \frac{\partial_y^2 J(x,y)}{\partial_y J(x,y)}}{\partial_y J(x,y)}=\frac{f''}{f'}(J(x,y)).$$
If the left hand side is a constant, i.e. does not depend on $x,y$, then $\frac{f''}{f'}$ is constant. Else let us consider a level $h$ of this left hand side. This defines an algebraic curve $\mathcal{C}_h$. As equation \eqref{eq1} does not admit a rational first integral, it has at most finitely many Darboux polynomials \cite{Jou}, and thus there exists $h\in\mathbb{C}$ such that $J_{\mid \mathcal{C}_h}$ is not constant. But then by construction $\frac{f''}{f'}(J_{\mid \mathcal{C}_h})$ is constant along the curve $\mathcal{C}_h$. As the image of $J_{\mid \mathcal{C}_h}$ is not reduced to a point, $\frac{f''}{f'}$ is constant.

Solving $\frac{f''}{f'}=h$, we conclude that either $f(z)=az+b$ or $f(z)=ae^{bz}+c$ with $a,b,c\in\mathbb{C}$. The first case implies that $\mathcal{F}$ itself should be a $k$-Darbouxian first integral, and thus that $ e^{\int K_1dx+K_2dy}$ should be algebraic. This is tested in step $8$, where residues of the differential form $ K_1dx+K_2dy$ are computed, and if all of them are rational and $ K_1dx+K_2dy$ has only simple poles, $\int K_1dx+K_2dy$ is then a $\mathbb{Q}$-linear combination of logs and thus the integrating factor will be algebraic. Step $8$ then returns an algebraic representation of it.

The second case implies that $\mathcal{F}$ is (up to affine transformation) the exponential of a Darbouxian first integral. The only possibility is then that $ \int e^{\int Gdx+Fdy} (-\frac{B}{A}dx+dy)$ integrates using a hyperexponential function. But then such a solution can be written $Re^{\int K_1dx+K_2dy}$ with $R$ rational. We can now plug this expression in the two equations that must be satisfied by $\mathcal{F}$
$$\partial_y^2 \mathcal{F}-K_2\partial_y \mathcal{F}=0,\quad \partial_x\mathcal{F}+F\partial_y \mathcal{F}=0$$
which gives the system
$$\partial_x R +F\partial_y R+(K_1+FK_2)R=0,\quad \partial_y^2 R+K_2\partial_y R +R \partial_y K_2=0$$
Such system can be solved in rational functions as an integrable connection \cite{barkatou2012computing}. It has at most a vector space of dimension $2$ of rational solutions. If it was $2$, remembering that $\mathcal{F}=1$ is a solution of the previous system, we would have a rational $W$ such that $We^{\int K_1dy+K_2dx}=1$. Such case with a rational integrating factor would have already been detected in step $8$ as $e^{\int K_1dy+K_2dx}$ would be rational, and thus $K_1dy+K_2dx$ would have simple poles with integer residues. Thus a most one rational solution up to constant factor exists. Step $9$ looks for a rational solution, and if one rational solution $R$ is found, returns the $y$ derivative of the Darbouxian first integral $\ln R +\int K_1dx+K_2dy$, i.e. a rational integrating factor. Else, the integrating factor $e^{\int K_1dy+K_2dx}$ we found before was minimal, and thus is returned.
\end{proof}

Remark that there are two possible FAIL reasons. Either the point $(x_0,y_0)$ has been badly chosen, either $\partial_x y=F(x,y)$ admits a rational first integral. Th FAIL answer (outside step $1$) can occur
\begin{itemize}
\item In step $5$, which implies that $P_1,P_3$ or $\hbox{num}(\partial_x(P_1/P_3)+F \partial_y(P_1/P_3)+P_1/P_3 \partial_y F)$ are Darboux polynomials, as they vanish on the solution but are not identically zero.
\item In step $7$, which implies that $G$ of $\hbox{num}(\partial_x(H/G)+F\partial_y(H/G)+H/G\partial_y F)$ are Darboux polynomials, as they vanish on the solution but are not identically zero.
\end{itemize}
Thus either the point $(x_0,y_0)$ has been chosen on a codimension set $1$ of the exceptional Darboux polynomials, or the equation admits a rational first integral. Outside of the spectrum of this first integral \cite{Cheze}, the degree of the minimal polynomial of its levels equals the degree of the first integral. Thus, if $\partial_x y=F(x,y)$ has no rational first integral of degree
$$\leq \max(N,2N+2d,2(2N+3d)+2d)=4N+8d$$
where $d$ is the max of numerator denominator degrees of $F$, these cases could only occur for exceptional Darboux polynomials. Thus Proposition \ref{propirr} still holds when weakening the hypothesis on the non existence of rational first integral: it is only necessary to check up to degree $4N+8d$. In practice, it is better to allow the FAIL answer to occur, and only if occurring for several initial conditions, to suspect existence of such rational first integral and perform the costly search for it.

Remark that for practical implementation, to compute the series $L_1$, it is faster to apply \underline{SolSeries} to an initial point $(x_0,y_0+\epsilon)$ and making all computation modulo $\epsilon^2$. This gives for the powers
$$(y(x)+\epsilon y_1(x))^n=y(x)^n+n\epsilon y_1(x) y(x)^{n-1}.$$
This $y_1(x)$ equals $\exp(\int \partial_y F_{\mid y=y(x)} dx)$, and thus taking the coefficient in $\epsilon$ divided by $n$ gives directly the series $L_1$. Thus the costly exponential of a series is avoided. Similarly, the integrand in $L_2$ can be computed faster using $Y$ and $L_1$ which contain the monomials in the numerator and denominator.\\

\noindent
\textbf{Examples}\\
Consider a first integral of the form
$$\mathcal{F}(x,y)=\ln P_1(x,y)+\sqrt{2}\ln \left(\frac{P_2(x,y)+\sqrt{2}}{P_2(x,y)-\sqrt{2}}\right)$$
If $P_1,P_2\in\mathbb{Q}(x,y)$, we see that the Galois action on $\sqrt{2}$ lets this expression invariant. Thus the equation
$$\partial_x y(x)= -\left(\frac{\partial_y \mathcal{F}}{\partial_x \mathcal{F}}\right)(x,y(x))$$
will have rational coefficients and admit $\mathcal{F}$ as a Darbouxian first integral. It will also admit an integrating factor
$$U=\frac{\partial_y P_1}{P_1}-\frac{4\partial_y P_2}{P_2^2-2}$$
We compare the timings of \underline{IntegratingFactor}, and a built in Maple function ``intfactor'' for random polynomials $P_1,P_2$
$$\begin{array}{|c|c|c|c|c|c|c|c|c|}\hline
\deg P_2+2\deg P_2      &  3  &  4  & 5  & 6  & 7  & 8   & 9  & 10  \\\hline
\hbox{IntegratingFactor}& 0.13&0.33 &0.91&2.47&3.7 & 9.5 &27.9& 53.4\\\hline
\hbox{intfactor}        & 0.03&0.047&0.08&0.08&0.14& 0.21& 0.2& 0.35\\\hline
\end{array}$$
Obviously, our computation time is not competitive. However, let us take a particular choice for $P_1,P_2$:
$$\mathcal{F}(x,y)=\ln (1+x^2y)+\sqrt{2}\ln \left(\frac{x^2+y\sqrt{2}}{x^2-y\sqrt{2}}\right)$$
The procedure ``intfactor'' does not find the integrating factor, as in fact this example does not fit its heuristic. The algorithm \underline{IntegratingFactor} however is guaranteed to find an integrating factor if one exists provided that $N$ is large enough. And indeed, after $0.6$s, it finds an integrating factor of degree $7$
$$U=\frac{(x^4+4x^2y-2y^2+4)x^2}{(x^2y+1)(x^4-2y^2)}$$

\subsection{Telescoper reduction}

Let us first remark that all telescoper forms given by Theorem \ref{thm1} have constant coefficients and more specialized forms when the first integral is more complicated. Looking at the $k$-Darbouxian case, we have
$$L I =(\partial_h y)^{\hbox{ord}(L)} H,\;\; L\in\mathbb{C}[\partial_h].$$
However, as $\partial_h y=U^{-1}$ and $(\partial_h y)^k\in \mathbb{K}(x,y)$, we can rewrite this relation (changing $H$)
$$L I =U^{-(\hbox{ord}(L) \hbox{ mod } k)} H,\;\; L\in\mathbb{C}[\partial_h].$$
Noting $H=P/Q$, and $l=\hbox{ord}(L)$ this gives
$$\sum\limits_{i=0}^{\lfloor l/k \rfloor} Q(x,y(x,h)) a_i \partial_h^{ki+(l\hbox{ mod } k)} I(x,h)=U^{-(l\hbox{ mod } k)}P(x,y(x,h))$$
The Liouvillian case can be recovered by putting $k=\infty$. The sum has then a single term. The Ricatti and differentially transcendental case similarly have only one term in the sum, as possible telescopers are $1,\partial_h$.

We see that when the sum has more than one term, the $a_i$ cannot be incorporated with $Q$, and thus the representation is non linear. This motivates us to introduce for the $k$-Darbouxian case a larger space of equations
\begin{defi}
We call a $k$-pseudo telescoper an equation of the form
\begin{equation}\label{eq6}
\sum\limits_{i=0}^{\lfloor l/k \rfloor} Q_i(x,y(x,h)) \partial_h^{ki+(l\hbox{ mod } k)} I(x,h)=U^{-(l\hbox{ mod } k)}P(x,y(x,h))
\end{equation}
with $Q_i,P\in\mathbb{C}[x,y]$.
\end{defi}

The interesting point is the space of $k$-pseudo telescopers for given $U,l$ of degree $\leq N$ is now a finite dimensional vector space. True telescopers form an algebraic subvariety of $k$-pseudo telescopers with the additional condition that $Q_i/Q_l\in\mathbb{C}$ which is not a linear condition. However we have

\begin{prop}\label{propred}
Assume equation \eqref{eq1} admits a $k$-Darbouxian first integral but not rational first integral. If a non trivial $k$-pseudo telescoper exists, then a true telescoper exists. The algorithm \underline{\sf ReduceTelescoper} always terminate and compute such telescoper. It runs in $\tilde{O}(N\hbox{ord}^{\omega+3} )$.
\end{prop}

Let us recall the notations
$$D_x=\partial_x+F\partial_y,\; D_h=U^{-1}\partial_y$$
With these notations, differentiating a function $f(x,y(x,h))$ in $x$ is equivalent to applying $D_x$ to $f$, and differentiating $f(x,y(x,h))$ in $h$ is equivalent to applying $D_h$ to $f$.\\

\noindent\underline{\sf ReduceTelescoper}\\
\textsf{Input:} A $k$-pseudo telescoper for the integral $I(x,h)=\int G(x,y(x,h)) dx$, the equation $\partial_x y=F(x,y)$ with its integrating factor $U$.\\
\textsf{Output:} A telescoper for the integral $I(x,h)=\int G(x,y(x,h)) dx$.\\
\begin{enumerate}
\item Note $\tilde{l}=\lfloor l/k \rfloor$ and $L_1=[Q_{\tilde{l}},\dots,Q_0,P]$.
\item $i=1$. While $\hbox{rank}_{\mathbb{K}(x,y)}((L_j)_{j=1\dots i})=i$ do
$$L_{i+1}=(D_x L_{i,j})_{j=1..\tilde{l}+1},D_x L_{i,\tilde{l}+2}-(l\hbox{ mod } k)L_{i,\tilde{l}+2}U^{-1}D_xU -U^{(l\hbox{ mod } k)}\sum_{j=1}^{\tilde{l}+1} L_{i,j} D_h^{kj+(l\hbox{ mod } k)} G $$
and increase $i=i+1$.
\item Build a row echelon form of the matrix $L$, and note $\tilde{Q}_r,\dots,\tilde{Q}_0,\tilde{P}$ the last non zero line where $\tilde{Q}_r$ is the first non zero coefficient. Return
$$\sum_{i=0}^r \frac{\tilde{Q}_i}{\tilde{Q}_r}\partial_h^{ki+(l\hbox{ mod } k)} I(x,h)=U^{-(l\hbox{ mod } k)} \frac{\tilde{P}}{\tilde{Q}_r}$$
\end{enumerate}

\begin{proof}
In step $2$, we compute $L_i$ until the $\hbox{rank}_{\mathbb{K}(x,y)}((L_j)_{j=1\dots i})\neq i$. As $L_j$ has always $\tilde{l}+2$ elements, the rank is $\leq \tilde{l}+2$, and thus the while loop terminates. Considering the pseudo telescoper
$$\sum\limits_{i=0}^{\lfloor l/k \rfloor} Q_i(x,y(x,h)) \partial_h^{ki+(l\hbox{ mod } k)} I(x,h)=U^{-(l\hbox{ mod } k)}P(x,y(x,h))$$
we can differentiate this relation in $x$, which gives
$$\sum\limits_{i=0}^{\lfloor l/k \rfloor} Q_i(x,y(x,h)) \partial_h^{ki+(l\hbox{ mod } k)} G(x,y(x,h))+
\sum\limits_{i=0}^{\lfloor l/k \rfloor} D_x(Q_i)(x,y(x,h)) \partial_h^{ki+(l\hbox{ mod } k)} I(x,h)$$
$$=U^{-(l\hbox{ mod } k)}\partial_x P(x,y(x,h))+\partial_x U^{-(l\hbox{ mod } k)} P(x,y(x,h))$$
and thus
$$\sum\limits_{i=0}^{\lfloor l/k \rfloor} D_x(Q_i)(x,y(x,h)) \partial_h^{ki+(l\hbox{ mod } k)} I(x,h)=$$
$$U^{-(l\hbox{ mod } k)}\left(D_x P-(l\hbox{ mod } k)PU^{-1}D_xU-
U^{(l\hbox{ mod } k)}\sum\limits_{i=0}^{\lfloor l/k \rfloor} Q_i D_h^{ki+(l\hbox{ mod } k)} G\right)(x,y(x,h))$$
The coefficients of this expression are given by $L_2$. This is also a pseudo telescoper for $\int G(x,y(x,h)) dx$. Remark that the coefficients are rational as the power of $U$ in the $D_h^{ki+(l\hbox{ mod } k)}$ equals $-(l\hbox{ mod } k)$ modulo $k$ and thus the factor $U^{(l\hbox{ mod } k)}$ exactly compensates.

Now this expression $L_2$ is a perfectly valid $k$-pseudo telescoper, and we can differentiate it again, and thus all lines of $L$ are $k$-pseudo telescopers for $\int G(x,y(x,h)) dx$. Thus so are their $\mathbb{K}(x,y)$-linear combinations. Thus the expression returned in step $3$ is a $k$-pseudo telescoper for $\int G(x,y(x,h)) dx$. Also, after loop of step $2$ stopped, a further derivative in $D_x$ can be written as linear combinations of the previous ones. Thus by induction, all $D_x$ derivatives can be written as linear combinations of these ones, and so the $k$-pseudo telescopers given by $L$ generate a $\mathbb{K}(x,y)$ vector space stable by $D_x$.

Let us now assume that the expression returned in step $3$ is not a true telescoper. As it as been normalized, this means that at least one of its coefficients $\tilde{Q}_i/\tilde{Q}_r$ is not constant. We now apply $D_x$ again on it. The dominant term is constant by assumption, and thus disappears. The $k$-pseudo telescoper obtained is thus of order strictly less than the one returned in step $3$. However, by construction, it is also a $\mathbb{K}(x,y)$-linear combination of the lines of $L$, and thus should have been returned by the row echelon form.

The only possible explanation is that this $k$-pseudo telescoper is $0$. Then $D_x(\tilde{Q}_i/\tilde{Q}_r)=0$ for $i=0\dots r$. As there are no rational first integral, this implies that $\tilde{Q}_i/\tilde{Q}_r$ is constant, and thus the returned  $k$-pseudo telescoper was a true telescoper.

Finally, for the computation time, each derivation $D_x$ and operator $D_h$ increases linearly the degree of numerators denominators of the coefficients of $L_i$. The operator $D_h$ is at most applied $\hbox{ord}$ times, and thus the degree of $L_i$ is $O(N\hbox{ord}\,i)$. Now we look at the degree of elements of the row echelon form of the matrix $L$. They can be written as sub-determinants and thus have degree at most $O(N\hbox{ord}^3)$. Now we know that polynomials products and sums can be done in quasi linear time, and the row echelon form will cost $O(\hbox{ord}^{\omega})$. Thus the total cost will be $\tilde{O}(N\hbox{ord}^{\omega+3})$.
\end{proof}

Remark that again the assumption on the non existence of rational first integrals can be weakened. The final line of the row echelon form is a linear combination of the lines of $L$, and the coefficients can be obtained through sub-determinants. Thus the degrees of the coefficients $\tilde{Q}_i/\tilde{Q}_r$ are bounded by an expression in $O(N\hbox{ord}^3)$. In practice, it is only useful to suspect the existence of such a large degree rational first integral after the algorithm has failed.

\subsection{Telescoper checking}

Given a $k$-pseudo telescoper candidate, we can reduce it through \underline{\sf ReduceTelescoper} as if it is a correct $k$-pseudo telescoper. We obtain then an expression. If this has not the form of a telescoper, then for sure the candidate was not good. Else we obtain a telescoper candidate, and we need a procedure to check if it is really a telescoper. The same holds in the Liouvillian, Ricatti and differentially transcendental case for which we will also have telescoper candidates to check.

\begin{lem}\label{lemrcheck}
The equation $I(x,h)=H(x,y(x,h))$ is a telescoper if and only if $D_x H=G$. The equation $\partial_h I(x,h)=H(x,y(x,h)) \partial_h y$ is a telescoper if and only if $\partial_y G=D_x H + H \partial_y F$.
\end{lem}

\begin{proof}
Differentiating the equation $I(x,h)=H(x,y(x,h))$ in $x$ gives
$$G(x,y(x,h))=\partial_x (H(x,y(x,h))$$
As $y(x,h)$ is transcendental in $x$, this is equivalent to $D_x H=G$. Conversely, if this is satisfied then
$$I(x,h)=H(x,y(x,h)) +f(h)$$
for some function $f(h)$. As $I(x,h)$ is defined up to addition of an arbitrary function in $h$, we can assume $f(h)=0$, and thus the telescoper equation is correct.

Differentiating the equation $\partial_h I(x,h)=H(x,y(x,h)) \partial_h y$ in $x$ gives
$$\partial_h G(x,y(x,h))=\partial_x(H(x,y(x,h))) \partial_h y+ H(x,y(x,h)) \partial_h\partial_x y$$
$$ \partial_y G(x,y(x,h)) \partial_h y=\partial_x(H(x,y(x,h))) \partial_h y+ H(x,y(x,h)) \partial_h y \partial_y F (x,y(x,h))$$
$$ \partial_y G(x,y(x,h))=\partial_x(H(x,y(x,h))) + H(x,y(x,h)) \partial_y F(x,y(x,h))$$
This is an equality over rational fractions in $x,y(x,h)$, and as $y(x,h)$ is transcendental in $x$, this is equivalent to
$$ \partial_y G=D_x H + H \partial_y F$$
Conversely, if this is satisfied then
$$\partial_h I(x,h)=H(x,y(x,h)) \partial_h y +f(h)$$
for some function $f(h)$. As $I(x,h)$ is defined up to addition of an arbitrary function in $h$, we can assume $f(h)=0$, and thus the telescoper equation is correct.
\end{proof}

Remark that in the Ricatti and differentially transcendental case, the operator $D_h$ is not defined as there is no integrating factor $U$. And indeed, checking a telescoper $1$ and $\partial_h$ does not require the existence or the knowledge of an integrating factor.\\

In the Ricatti and differentially transcendental case, obviously we cannot replace $\partial_h y$ by its expression in function of the integrating factor as it does not exist. However, for a more uniform treatment of this case, we will define a function $u(x,y)$ with $\partial_h y=u(x,y)^{-1}$, and note as before $D_h=u(x,y)^{-1} \partial_y$. Thus the possible telescopers write
$$I(x,h)=H(x,y),\quad \partial_h I(x,h)=u(x,y)^{-1} H(x,y)$$
For the second case, applying $D_x$, we obtain
$$u(x,y)^{-1} \partial_y G=-u(x,y)^{-2} D_x u(x,y) H(x,y)+ u(x,y)^{-1} D_x H(x,y) $$
and thus
$$\partial_y G=-u(x,y)^{-1} D_x u(x,y) H(x,y)+ D_x H(x,y) $$
Now the condition of Lemma \ref{lemrcheck} can be recovered by substituting $D_xu(x,y)=-u(x,y) \partial_y F$. Thus from now on, in the Riccati/differentially transcendental case, we will write for the telescoper $\partial_h$
$$\partial_h I(x,h)= U^{-1} H(x,y)$$
where $U=u(x,y)$, a non zero function satisfying $D_xu(x,y)=-u(x,y) \partial_y F$.\\

\noindent\underline{\sf CheckTelescoper}\\
\textsf{Input:} A telescoper for the integral $\int G(x,y(x,h)) dx$, the equation $\partial_x y=F(x,y)$ with optionally an integrating factor $U$.\\
\textsf{Output:} True or False\\
\begin{enumerate}
\item  Note $T$ the telescoper part and $H$ the certificate part. Compute the quantity
$$K=U^{-(l\hbox{ mod } k)}(D_x H-(l\hbox{ mod } k)HU^{-1}D_x U)-\sum_{i=0}^l a_i D_h^i G$$
where $\sum_{i=0}^l a_i\partial_h^i=T$, $a_l=1$, $k$ is the algebraic extension of $U$ or $k=\infty$ if $U$ is transcendental and $U=u(x,y)$ if it is not given.
\item Replace $\partial_x u(x,y)=-F\partial_y u(x,y)-u(x,y) \partial_y F$.
\item Return check equality $K=0$ as an an element of $\mathbb{K}(x,y,U,u(x,y))$.
\end{enumerate}

\begin{prop}
The algorithm \underline{\sf CheckTelescoper} correctly checks the validity of a telescoper.
\end{prop}

\begin{proof}
If $U$ is not given, then telescopers are only possible for $T=1,\partial_h$. For $T=1$, we have for the quantity $K=D_x H-G$, which is indeed the test of Lemma \ref{lemrcheck}. For $T=\partial_h$, we have for the quantity
$$K=U^{-1}(D_x H-HU^{-1}D_x U)-U^{-1} D_y G$$
Now step $2$ makes the substitution $U=u(x,y)$ and simplifies by the relation $\partial_x u(x,y)=-F\partial_y u(x,y)-u(x,y) \partial_y F$, which is equivalent to $D_xu(x,y)=-u(x,y) \partial_y F$. Thus we have
$$K=U^{-1}(D_x H+H\partial_y F-D_y G)$$
and so testing $K=0$ is equivalent to the test of Lemma \ref{lemrcheck}.

Else $U$ is given. If the telescoper is correct, then we have
$$\sum\limits_{i=0}^{\lfloor l/k \rfloor} a_i \partial_h^{ki+(l\hbox{ mod } k)} I(x,h)=U^{-(l\hbox{ mod } k)}H(x,y(x,h)).$$
We can differentiate this relation in $x$, which gives
$$\sum\limits_{i=0}^{\lfloor l/k \rfloor} a_i \partial_h^{ki+(l\hbox{ mod } k)} G(x,y(x,h))=U^{-(l\hbox{ mod } k)}(\partial_x H(x,y(x,h))+H(x,y(x,h))U^{-1}\partial_x U)$$
and thus
$$\sum\limits_{i=0}^{\lfloor l/k \rfloor} a_i D_h^{ki+(l\hbox{ mod } k)} G=U^{-(l\hbox{ mod } k)}(D_x H+HU^{-1}D_x U)$$
which is the tested equality in step $3$. Conversely, it this equality is true, then we can integrate in $x$, giving
$$\sum\limits_{i=0}^{\lfloor l/k \rfloor} a_i \partial_h^{ki+(l\hbox{ mod } k)} I(x,h)=U^{-(l\hbox{ mod } k)}H(x,y(x,h)) +f(h)$$
The integrating constant $f(h)$ can be removed by subtracting to $I(x,h)$ a function $g(h)$ solution of the equation
$$\sum\limits_{i=0}^{\lfloor l/k \rfloor} a_i \partial_h^{ki+(l\hbox{ mod } k)} g(h)=f(h).$$
As $I(x,h)$ is defined up to an arbitrary function of $h$, this implies that the telescoper is correct.
\end{proof}

\subsection{Telescoper search}

The approach is similar as for finding symbolic first integrals. We compute a series solution at high order, search for a pseudo telescoper, and if found reduce it to a true telescoper. Then a check occurs to verify that this candidate is valid. Using a proof that it will be valid for a generic initial point will confirm the correctness of the algorithm.\\

\noindent\underline{\sf FindTelescoper}\\
\textsf{Input:} Rational functions $F,G\in\mathbb{K}(x,y)$, integers $\hbox{ord},N$, an initial point $(x_0,y_0)$ and optionally an integrating factor $U$.\\
\textsf{Output:} A telescoper or ``None'', or FAIL.
\begin{enumerate}
\item Note $M=\tfrac{1}{2}(N+1)(N+2)(\lceil \hbox{ord}/k \rceil +2)$ where $k=\hbox{ord}+1$ if $U$ is not given or transcendental, and else $k$ equals the algebraic degree of $U$. If $F(x_0,y_0),U(x_0,y_0)$ is singular or $U(x_0,y_0)=0$, return FAIL.
\item Compute the list $LG$ of $(D_h^j G)_{j=0\dots \hbox{ord}}$ where $D_h=u(x,y)^{-1} \partial_y$ and simplify using $\partial_y u(x,y)=(U^{-1} \partial_y U) u(x,y)$ if $U$ is provided.
\item Compute $S=$\underline{SolSeries}$(F,(x_0,y_0+\epsilon),M,\max(d,N)+1)$ where $d$ is the max of degrees in $x,y$ (neglecting $u(x,y)$ if present) numerators and denominators of $LG$ and \underline{SolSeries} is computed mod $\epsilon^2$.
\item Compute $\tilde{S}$ the list of lists of $y(x)^i (\partial_h y)^j,\; j=0\dots \hbox{ord}, i=0\dots \max(d,N)$ using the fact that $y(x)^i\partial_h y$ is the $\epsilon$ coefficient of $S$.
\item Compute $J$ the list of lists of $y(x)^i\int_0^x (LG_j(x,y(x)))_{\mid u(x,y)^{-1}=\partial_h y} dx, j=0\dots \hbox{ord}, i=0\dots \max(d,N)$
\item For $l$ from $0$ to $\hbox{ord}$ do
\begin{enumerate}
    \item Consider the expression
    $$\!\!\!\!\!\!\!\!\!\!\!\!\!\!\!\!\!\!\!\!\!\!\!\!\sum\limits_{i=0}^{\lfloor l/k \rfloor} Q_i(x,y(x)) 
    \int (LG_{ki+(l\hbox{ mod } k)})_{\mid u(x,y)^{-1}=\partial_h y} dx
    -(\partial_h y)^{-(l\hbox{ mod } k)}P(x,y(x))+
    \delta \tilde{P}(x,y(x))$$
    where $\delta=0$ if $l$ multiple of $k$ and else $1$, and $Q_i,P,\tilde{P}$ are polynomials of degree $N$ with undetermined coefficients.
    \item Replacing the integrals by their series $J$, solve this as linear system equal $O(x^{M+1})$, and if a non trivial solution found, substitute it to define the $k$ pseudo telescoper $T$ (forgetting the $\tilde{P}$)
    $$\sum\limits_{i=0}^{\lfloor l/k \rfloor} Q_i(x,y(x,h)) \partial_h^{ki+(l\hbox{ mod } k)} I(x,h)=U^{-(l\hbox{ mod } k)}P(x,y(x,h)).$$
    Else skip the next steps (c),(d).
    \item If $U$ is algebraic, make in $T$ the substitution $Q_i \rightarrow \beta^{ki+(l\hbox{ mod } k)} Q_i$ where $\beta=U(x_0,y_0)$.
    \item $T=$\underline{ReduceTelescoper}$(T,G,F,U)$. If \underline{CheckTelescoper}$(T,G,F,U)$ return $T$ else return FAIL.
\end{enumerate}
\item Return ``None''.
\end{enumerate}

\begin{prop}
Assume equation $y'=F(x,y)$ does not admit a rational first integral, and assume $U$ is provided and minimal if it exists. If $\int G(x,y(x)) dx$ admits a telescoper of order $\hbox{ord}$ and certificate degree $\leq N$, then algorithm \underline{\sf FindTelescoper} returns either a correct telescoper, or FAIL. If algorithm \underline{\sf FindTelescoper} returns ``None'', then no telescoper of order $\leq \hbox{ord}$ and certificate degree $\leq N$ and with structure according to given $U$ exists. If algorithm \underline{\sf FindTelescoper} returns FAIL, then $(x_0,y_0)$ belongs to a codimension $1$ algebraic set. When Pade Hermite is used for step $6b$, the complexity is $\tilde{O}(N^{\omega+1} \hbox{ord}^{\omega-1}+N\hbox{ord}^{\omega+3})$.
\end{prop}

\begin{proof}[Proof of Theorem \ref{thm2}]
By construction, if a telescoper is returned, it is correct as its correctness is checked in step 6d. Remark now that multiplying the integrating factor $U$ by a constant $\beta\in\mathbb{C}^*$ transforms a $k$ pseudo telescoper by multipying $\partial_h^i I(x,h)$ by $\beta^i$. Thus we can assume that $U(x_0,y_0)=1$ in the computation until step $6c$ where the coefficients of the candidate $T$ are scaled back accordingly to the value of $U(x_0,y_0)$. When $U$ is not algebraic, by Galois action, we can always assume that $U(x_0,y_0)=1$. The only exception is when $U(x_0,y_0)=0$, which case is discarded in step $1$ and happens on a codimension $1$ algebraic set. This assumption insures that $\partial_h y(x,0)$ is correctly calculated in steps $3,4$ where it is assumed that $\partial_h y(x_0,0)=1$, which does require $U(x_0,y_0)=1$ due to relation $U(x,y(x))=\partial_h y(x,0)$.

Let us now assume that the returned answer is ``None''. This can only happen in step 7, when for all $l=0\dots \hbox{ord}$, the linear systems 
\begin{equation}\begin{split}\label{eqser}
\sum\limits_{i=0}^{\lfloor l/k \rfloor} Q_i(x,y(x)) \int (LG_{ki+(l\hbox{ mod } k)})_{\mid u(x,y)^{-1}=\partial_h y} dx\\
-(\partial_h y)^{-(l\hbox{ mod } k)}P(x,y(x))+\delta \tilde{P}(x,y(x))=O(x^{M+1})
\end{split}\end{equation}
have no (non zero) solutions.
By construction in step 2, we have
$$(LG_{ki+(l\hbox{ mod } k)})_{\mid u(x,y)^{-1}=\partial_h y(x,h), y=y(x,h)}=\partial_h^{ki+(l\hbox{ mod } k)} G(x,y(x,h))$$
Thus commuting the derivative and integral, and evaluating at $h=0$ gives
$$\int_0^x (LG_{ki+(l\hbox{ mod } k)})_{\mid u(x,y)^{-1}=\partial_h y(x,0),y=y(x)} dx= (\partial_h^{ki+(l\hbox{ mod } k)} I(x,h))_{\mid h=0} +C_i$$
where $C_i$ is an unknown integration constant.

If a telescoper existed, then we would have a relation of the form
$$\sum\limits_{i=0}^{\lfloor l/k \rfloor} a_i(\partial_h^{ki+(l\hbox{ mod } k)} I(x,h))_{\mid h=0}-(\partial_h y)^{-(l\hbox{ mod } k)}H(x,y(x))=0$$
Multiplying this equation by the denominator of $H=P/Q$, we would get
$$\sum\limits_{i=0}^{\lfloor l/k \rfloor} a_i Q(x,y(x)) (\partial_h^{ki+(l\hbox{ mod } k)} I(x,h))_{\mid h=0}-(\partial_h y)^{-(l\hbox{ mod } k)}P(x,y(x))=0$$
$$\sum\limits_{i=0}^{\lfloor l/k \rfloor} a_i Q(x,y(x)) \left( \int_0^x (LG_{ki+(l\hbox{ mod } k)})_{\mid u(x,y)^{-1}=\partial_h y(x,0),y=y(x)} dx-C_i\right)$$
$$-(\partial_h y)^{-(l\hbox{ mod } k)}P(x,y(x))=0$$
$$\sum\limits_{i=0}^{\lfloor l/k \rfloor} a_i Q(x,y(x)) \int_0^x (LG_{ki+(l\hbox{ mod } k)})_{\mid u(x,y)^{-1}=\partial_h y(x,0),y=y(x)} dx$$
$$-(\partial_h y)^{-(l\hbox{ mod } k)}P(x,y(x))-\sum\limits_{i=0}^{\lfloor l/k \rfloor} a_i Q(x,y(x))C_i=0$$
When $l\hbox{ mod } k=0$, the last sum can be combined with the second part, and when $l\hbox{ mod } k\neq 0$, we have $\delta=1$ in relation \eqref{eqser}. In both cases, such relation would give a non trivial solution of equation \eqref{eqser}.  Thus if the algorithm returns ``None'', there is no telescoper of order and degree less than $\hbox{ord},N$.

Let us now assume that the FAIL answer is returned. This answer is returned in step 6d, when the telescoper $T$ found is not a correct telescoper. This means that in step 6c, the $k$ pseudo telescoper found $T$ was not correct.
Let us now consider the system
\begin{equation}\label{eqsys}
\left\lbrace \begin{array}{c}
\partial_x y(x)= F(x,y(x))\\
\partial_x I_0(x)= G(x,y(x))\\
\partial_x I_i(x)= (D_h^i(G))_{y=y(x)},\quad i=1\dots l
\end{array}\right. 
\end{equation}
This is a rational vector field in dimension $l+2$. Considering the initial condition $(x_0,y_0,0,\dots,0)$, the solution is
$$y(x,0),\int_0^x (\partial_h^i G(x,y(x,h)))_{\mid h=0} dx$$
Thus the series computed in step $J$ for $i=0$ are the series solutions of \eqref{eqsys}. Using the theorem of generalized extatic curves of \cite{cheze2020symbolic}, we thus know that equation \eqref{eqser} can be satisfied for a generic initial condition $x_0,y_0,z_0,\dots,z_l$ if and only if it is satisfied at any order.

Remark now that changing the initial condition for system \eqref{eqsys} to $x_0,y_0,z_0,\dots,z_l$ only adds the constants to the solution
$$y(x,0),z_i+\int_0^x (\partial_h^i G(x,y(x,h)))_{\mid h=0} dx$$
These constants then change the relation \eqref{eqser} to
$$\sum\limits_{i=0}^{\lfloor l/k \rfloor} Q_i(x,y(x)) \int (LG_{ki+(l\hbox{ mod } k)})_{\mid u(x,y)^{-1}=\partial_h y} dx-(\partial_h y)^{-(l\hbox{ mod } k)}P(x,y(x))+$$
$$\delta \tilde{P}(x,y(x))+\sum\limits_{i=0}^{\lfloor l/k \rfloor} Q_i(x,y(x)) z_{ki+(l\hbox{ mod } k)}=O(x^{M+1})$$
This is again the same type of relation where $\tilde{P}$ is replaced by
$$\tilde{P}+\sum\limits_{i=0}^{\lfloor l/k \rfloor} Q_i(x,y(x)) z_{ki+(l\hbox{ mod } k)}.$$
Thus if a relation of the form \eqref{eqser} is found for an initial condition $(x_0,y_0,0,\dots,0)$ with generic $x_0,y_0$, then a solution exists for a generic initial condition $x_0,y_0,z_0,\dots,z_l$. Thus the generalized extatic curve result applies, and so a solution of \eqref{eqser} for a generic $x_0,y_0$ (so outside a codimension $1$ algebraic set $\Sigma_{\mathcal{E}}$) implies that
$$\sum\limits_{i=0}^{\lfloor l/k \rfloor} Q_i(x,y(x))(\partial_h^{ki+(l\hbox{ mod } k)} I(x,h)_{\mid h=0}+C_i)-(\partial_h y)^{-(l\hbox{ mod } k)}P(x,y(x))+\delta \tilde{P}(x,y(x))=0$$
If $y(x)$ is transcendental, by Galois action this relation is valid for any $h$
\begin{equation}\label{eqpseudok}
\sum\limits_{i=0}^{\lfloor l/k \rfloor} Q_i(x,y(x,h))(\partial_h^{ki+(l\hbox{ mod } k)} I(x,h)+C_i)-(\partial_h y)^{-(l\hbox{ mod } k)}P(x,y(x,h))+\delta \tilde{P}(x,y(x,h))=0
\end{equation}
A planar vector field without rational first integrals has finitely many Darboux polynomials \cite{Jou}, and thus all of them define a codimension $1$ algebraic set $\Sigma_{\mathcal{D}}$. So for $(x_0,y_0)\notin \Sigma_{\mathcal{E}} \cup \Sigma_{\mathcal{D}}$, relation \eqref{eqpseudok} holds.

If $\delta=1$, then $l \hbox{ mod } k\neq 0$. Thus we have $U$ algebraic and not rational. Thus there is a Galois action $h \rightarrow \xi h$ (where $\xi$ is a $k$ root of unity) which can act on \eqref{eqpseudok}. All the terms are multiplied by the same power of $\xi$, except for $\tilde{P}(x,y(x,h))$. This gives us new relations, and by linear combination of two such relations, we deduce that $\tilde{P}(x,y(x,h))=0$, which implies that $\tilde{P}=0$ (for $(x_0,y_0) \notin \Sigma_{\mathcal{E}} \cup \Sigma_{\mathcal{D}}$). Thus we have the relation
\begin{equation}\label{eqpseudok2}
\sum\limits_{i=0}^{\lfloor l/k \rfloor} Q_i(x,y(x,h))(\partial_h^{ki+(l\hbox{ mod } k)} I(x,h)+C_i)-(\partial_h y)^{-(l\hbox{ mod } k)}P(x,y(x,h))=0
\end{equation}

Now we apply the same reduction procedure \underline{ReduceTelescoper} to the relation \eqref{eqpseudok2}. When differentiating in $x$, we still have the same relation
$$\partial_x (\partial_h^{ki+(l\hbox{ mod } k)} I(x,h)+C_i) = \partial_h^{ki+(l\hbox{ mod } k)} G$$
whatever the constants $C_i$ are. Thus, we obtain after reduction of order and diving by dominant coefficient a relation of the form
\begin{equation}\label{eqpseudok3}
\sum\limits_{i=0}^{\lfloor l/k \rfloor} a_i(\partial_h^{ki+(l\hbox{ mod } k)} I(x,h)+C_i)-(\partial_h y)^{-(l\hbox{ mod } k)}H(x,y(x,h))=0
\end{equation}
where the coefficients $a_i$ and certificate $H$ are the same as the telescoper $T$ obtained in step (6d).

Now the telescoper $T$ in step (6d) (after reduction) being correct is equivalent to have the relation
$$\sum\limits_{i=0}^{\lfloor l/k \rfloor} Q_i(x,y(x,h))(\partial_h^{ki+(l\hbox{ mod } k)} I(x,h))-(\partial_h y)^{-(l\hbox{ mod } k)}P(x,y(x,h))=0$$
up to adding to the integral $I(x,h)$ an arbitrary function of $h$. We see that adding to $I(x,h)$ a suitable exponential polynomial in $h$, any set of constants $C_i$ can be obtained. Thus the telescoper $T$ in step (6d) is correct, and thus FAIL answer does not occur. All these previous reasoning are correct if $(x_0,y_0) \notin \Sigma_{\mathcal{E}} \cup \Sigma_{\mathcal{D}}$. So FAIL answer can only happen when $(x_0,y_0)\in \Sigma_{\mathcal{E}} \cup \Sigma_{\mathcal{D}}$ or $F(x_0,y_0),U(x_0,y_0)$ is singular or $U(x_0,y_0)=0$. This is an algebraic set of codimension $1$ as required.

For the computation cost, steps $2,3,4,5$ compute $O(N\hbox{ord})$ series at order $O(N^2)$. Products and solving differential equation \eqref{eq1} cost $\tilde{O}(N^2)$, and thus total cost is $\tilde{O}(N^3\hbox{ord})$. Step 6b can be done using Pade Hermite: this equation can be seen as searching a linear combination of the $y(x)^i\int_0^x (LG_j(x,y(x)))_{\mid u(x,y)^{-1}=\partial_h y} dx,\; y(x)^i (\partial_h y(x))^j$ with series in $x$. The cost is $\tilde{O}(d^{\omega-1} \sigma)$ where $d$ is the number of terms and $\sigma$ the precision. here we have $d=O(N \hbox{ord})$ and $\sigma=O(N^2)$. Thus the cost is $\tilde{O}(N^{\omega+1} \hbox{ord}^{\omega-1})$. The reduction step 6d then cost according to Proposition \ref{propred} $O(N\hbox{ord}^{\omega+3})$ and thus total cost is $\tilde{O}(N^{\omega+1} \hbox{ord}^{\omega-1}+N\hbox{ord}^{\omega+3})$
\end{proof}

One can wonder what was the utility of $\tilde{P}$ as it is discarded at step 6b. We know that for a suitable choice of integration constants for $I(x,h)$, we can discard it. So we can discard it for a suitable choice of integration constants in $\partial_h^i I(x,h)_{\mid h=0}$, however we do not know this suitable choice beforehand. Adding constants to $\partial_h^i I(x,h)_{\mid h=0}$ makes potentially appear a linear combination of the $Q_i$, and as we want to keep the system solving linear, we are forced to introduce the $\tilde{P}$ only to take into account these unknown integration constants.

\section{Existence of a telescoper}

\subsection{The logarithm foliation}
For some foliations, it is possible to classify the possible telescopers, and then decide the existence of it for a given integral. One of the most interesting examples is the case
$$y=\ln x,\;\; \partial_x y= \tfrac{1}{x}, \;\; U=1$$

For this foliation, we can precise the possible form of integrals admitting a telescoper.

\begin{prop}\label{propex}
An integral of $G\in\mathbb{C}(x,\ln x)$ admitting a telescoper can be written
\begin{equation}\label{eqintln}
\sum_{p\in\mathbb{Z}^*,\lambda\in \mathbb{C}}\!\!\!\! a_{p,\lambda}\hbox{Ei}(p\ln x +\lambda)+
\!\!\!\!\!\!\sum_{p,r\in\mathbb{N}^*,\lambda\in \mathbb{C}^*}\!\!\!\!\!\!
b_{p,r,\lambda}(\ln x)^r \hbox{Li}_p(\lambda x)+\sum_{\lambda\in \mathbb{C}} \lambda \ln (K_{\lambda}(x,\ln x))+H(x,\ln x)
\end{equation}
where $K_\lambda,H$ are rational functions.
\end{prop}

\begin{proof}
We know that if $\int G(x,\ln x)$ admits a telescoper, this integral will be a linear combination of integrals of the form
$$\int e^{\kappa(y-\ln x)} (\ln x)^r (R_1(x,y)dx+R_2(x,y)dy)$$
with $r\in\mathbb{N}, \kappa\in\mathbb{C}$. Let us consider a pole (defining a curve $\mathcal{C}$) of this $1$ form which is not $x=c$ nor $y=c$. Then along this curve $\mathcal{C}$, we look at the residue in $x$ of the $1$ form. The residue is of the form
$$ \sum\limits_{\kappa \in S} \sum_{r \in \mathbb{N}} a_{\kappa,r}(x,\alpha(x)) e^{\kappa \alpha(x)} x^{-\kappa} (\ln x)^r$$
where $\alpha(x)$ is the algebraic function giving $y$ in function of $x$ on $\mathcal{C}$, and $a_{\kappa,r}$ polynomials. This residue should be constant in $x$ for the $1$ form to be closed. As $\alpha(x)$ is non constant, there can be no compensations between the terms, and each term is transcendental except $\kappa=0,r=0$. Thus the only possibility is that all $a_{\kappa,r}=0$ on $\mathcal{C}$ except possibly $a_{0,0}$. So all the forms are exact along this curve, except possibly for $\kappa=0,r=0$. This case corresponds to a rational closed $1$ form, which can thus be integrated using a log.

For the curve $y=c$, the residue in $x$ is of the form
$$ \sum\limits_{\kappa \in S} \sum_{r \in \mathbb{N}} a_{\kappa,r}(x,c) e^{\kappa c} x^{-\kappa} (\ln x)^r$$
Compensations can occur between terms for which $\kappa-\kappa'\in\mathbb{Z}$ and with the same $r$. However, only the terms with $\kappa\in\mathbb{Z}$ and $r=0$ can give a constant non zero function. These correspond to the integrals
\begin{equation}\label{eqex1}
\int \frac{e^{\kappa y}}{y-c} dy
\end{equation}

For the curve $x=0$, the residue in $y$ is of the form
$$ \sum\limits_{\kappa \in S} \sum_{r \in \mathbb{N}} a_{\kappa,r}(c,y) e^{\kappa y} c^{-\kappa} (\ln c)^r$$
No compensations can occur between terms with distinct $\kappa$'s. But for non zero $\kappa$ only transcendental functions can be obtained or $0$. Thus these cases will give exact forms along this curve. For $\kappa=0$, all terms are polynomials in $y$. The integrals for $\kappa=0$ correspond to the integrals
\begin{equation}\label{eqex2}
\int \frac{(\ln x)^r}{x-c} dx.
\end{equation}

At last, if the integral has no residue along a finite curve, then the $R_1,R_2$ are polynomials. These $1$ forms can be integrated with elementary functions, using here $e^{\kappa y},x^{-\kappa},\ln x, \ln \ln x$. These are included in the expression \eqref{eqintln}.

Thus, any $1$-form appearing in the expression of $\int G(x,\ln x)$ is, up to a linear combination of integrals \eqref{eqex1},\eqref{eqex2}, the integral of a closed rational $1$ form and integrals with polynomials $R_1,R_2$, an exact form. The integrals \eqref{eqex1},\eqref{eqex2} integrate using $Ei,Li$ special functions and logs, giving the expression \eqref{eqintln}.
\end{proof}



\begin{prop}\label{propbound}
If $f(x,\ln x)$ with  $f\in \mathbb{C}(x,y)$ admits a telescoper, then the $H$ in expression \eqref{eqintln} is such that $\hbox{den}(H)$ has the same factors as $\hbox{den}(f)$ with multiplicity at most one less and $\deg_x H \leq \max(0,\deg_x f +1),\; \deg_y H \leq \max(0,\deg_y f +1)$.\\
If expression \eqref{eqintln} does not contain a $H$ term, then the integrand has only simple poles except possibly at $x=0$. Conversely, if the integrand $f$ has only simply poles except possibly at $x=0$, then $H\in \mathbb{C}[x,\tfrac{1}{x},\ln x]$.\\
\end{prop}

\begin{proof}
We look at the series expansion in $x$ at infinity of expression \eqref{eqintln}. Assume it grows at least as $x$, the $K$ part nor $Li$ part plays any role in the dominant term. After differentiation, the $Ei$ terms give a dominant term of the form
$$x^{p-1}\sum_{\lambda\in\mathbb{C}} \frac{r_\lambda}{\ln x-\lambda}$$
and the $\partial_x H$ part gives a term of the form
$$\partial_x (H(x,\ln x))=x^{p-1}(p w(\ln x)+w'(\ln x)) +o(x^{p-1}),\quad w\in\mathbb{C}(y).$$
So by construction, any pole of the $\partial_x H$ part will be multiple, and any pole of the $Ei$ part will be simple. Thus no compensation is possible. Thus we have $\deg_x H \leq \max(0,\deg_x f +1)$.

We now look at the series expansion in $h$ of \eqref{eqintln} when replacing $\ln x$ by $\ln x +h$. Assume it grows at least as $h$, the $K$ part nor $Ei$ part plays any role in the dominant term. After differentiation, the $Li$ terms give a dominant term of the form
$$h^p\sum_{\lambda\in\mathbb{C}} \frac{r_\lambda}{x-\lambda}$$
and the $\partial_x H$ part give a term of the form
$$\partial_x (H(x,\ln x+h))=h^p w'(x) +o(h^p),\quad w\in\mathbb{C}(y).$$
So by construction, any pole of the $H$ part will be multiple, and any pole of the $Li$ part will be simple. Thus no compensation is possible. Thus we have $\deg_y H \leq \max(0,\deg_y f +1)$.

Finally, any pole of $H$ gives a pole of $\partial_x H+ \tfrac{1}{x} \partial_y H$ as the vector field is never tangent to any algebraic curve, and for the singular curve $x=0$, we have
$$(\partial_x + \tfrac{1}{x} \partial_y)( x^{-k}w(y))=x^{-k-1}(-kw(y)+w'(y))+o(x^{-k-1})$$
which can never simplify as $w$ should be a rational function. This gives the statement about $H$.

For the second statement, differentiating in $x$ expression \eqref{eqintln} gives simple poles except for the $Ei$ parts for which
$$\partial_x Ei(p\ln x + \lambda)= x^{p-1}\sum_{\lambda\in\mathbb{C}} \frac{r_\lambda}{\ln x-\lambda}$$
and thus substituting $\ln x=y$ gives a rational fraction with a multiple pole at $x=0$. Conversely, if $f$ has only simple poles outside $x=0$, we can apply the first part of the Proposition: the rational fraction $H$ cannot have any poles except the multiple ones, so here only $x=0$. Thus $H\in\mathbb{C}[x,\tfrac{1}{x},\ln x]$.

\end{proof}

\begin{prop}\label{propelem}
A fraction in $f\in \mathbb{C}(x,y)$ with only simple poles except possibly at $x=0$ such that $\int f(x,\ln x) dx$ is elementary can be written
$$ f(x,y)= \sum_{\lambda\in\mathbb{C}} \frac{r_\lambda}{x(y-\lambda)}+ \sum_{\lambda\in\mathbb{C}} \frac{s_\lambda}{x-\lambda} + g(x,y) $$
where $g$ has no pole along a line $x=\lambda,\lambda\neq 0$ or $y=\lambda,\; \lambda\in\mathbb{C}$.
\end{prop}

\begin{proof}
We assume that $\int f(x,\ln x) dx$ is elementary and has no multiple poles outside $x=0$. Computing the partial fraction decomposition in $x$ of $f$ will give for poles of the form $x-\lambda,\; \lambda\neq 0$ an expression of the form
$$\frac{s_\lambda(y)}{x-\lambda}$$
Now replacing $y$ by $\ln x+h$, we obtain either a residue which depends on $h$ (this is forbidden for elementary integrals), or $s_\lambda$ is constant (these terms are allowed by the second sum).
Now computing the partial fraction decomposition in $y$ of $f$ will give for poles of the form $y-\lambda$ an expression of the form
$$\frac{w_\lambda(x)}{y-\lambda}$$
Now replacing $y$ by $\ln x+h$, we obtain either a residue which depends on $h$ (this is forbidden for elementary integrals), or $xw_\lambda(x)$ is constant (these terms are allowed by the first sum). Thus removing these terms from $f$ produces a function $g$ which has no pole along a line $x=\lambda,\lambda\neq 0$ or $y=\lambda,\; \lambda\in\mathbb{C}$.
\end{proof}

\noindent\underline{\sf IntegrateLnFoliation}\\
\textsf{Input:} A rational function $f\in\mathbb{C}(x,y)$\\
\textsf{Output:} If possible, an expression of the form \eqref{eqintln}, or ``None''.
\begin{enumerate}
\item  Note $F=P/Q$ with $Q$ the polynomial whose factors are those of $\hbox{den}(f)$ with multiplicity one less, and $P$ polynomial with undetermined coefficients with $\deg_x P \leq \max(\deg_x Q,\deg_x f+\deg_x Q +1),\; \deg_y P \leq \max(\deg_y Q,\deg_y f+\deg_y Q +1)$.
\item Look for $P$ such that $f-(\partial_x F+\tfrac{1}{x}\partial_y F)$ has only simple poles outside $x=0$. If none, return ``None''. Else note $\tilde{f}$ the resulting fraction.
\item Compute the residue in $y$ along poles of $\tilde{f}$ of the form $y=\lambda$. If they are in $\mathbb{C}[x,1/x]$, then compute the integral for these terms, and remove them from $\tilde{f}$. Else return ``None''.
\item Compute the residue in $x$ along poles of $\tilde{f}$ of the form $x=\lambda,\; \lambda\neq 0$. If they are in $\mathbb{C}[y]$, then compute the integral for these terms, and remove them from $\tilde{f}$. Else return ``None''.
\item Look for an integral of $\tilde{f}$ of the form
$$P(x,y)+\sum \lambda_i \ln f_i(x,y) $$
where $f_i(x,y)$ are the absolute factors of the denominator of $\tilde{f}$,and $P\in\mathbb{C}[x,\tfrac{1}{x},y], \; \deg P \leq \max(0,\deg \tilde{f})+2$ and the order of the pole at $x=0$ at most one less than the one in $\tilde{f}$. If one is found, return the expression for the integral, else return ``None''.
\end{enumerate}

\begin{proof}[Proof of correctness of algorithm \underline{\sf IntegrateLnFoliation}]
In steps $1,2$ we build the rational fraction $F$ with degree bounds of Proposition \ref{propbound}, and thus is $f$ has a telescoper, such $F$ should exist such that all multiple poles except possibly $x=0$ are removed.
In steps $3,4$, we compute the residues of $\tilde{f}$ along $y=\lambda$ and $x=\lambda$. If the poles in $1/(y-\lambda)$ do not have a coefficient of the form $\mathbb{C}[x,1/x]$, then they cannot be removed using either the $Ei$ functions nor the elementary part of \eqref{eqintln}. Similarly with the poles in $1/(x-\lambda)$ which cannot be remover if not of the form $\mathbb{C}[y]$.

Once these parts are removed, we thus have an integrand $\tilde{f}$ whose integral can be written \eqref{eqintln} but without poles along $y=\lambda$ and $x=\lambda,\lambda\neq 0$. We can decompose $\tilde{f}=\tilde{f}_1+\tilde{f}_2$ where integral of $\tilde{f}$ is a linear combination of the $Ei,Li$, and integral of $\tilde{f}_2$ is elementary. By Proposition \ref{propelem}, we can write
$$ \tilde{f}_2(x,y)= \sum_{\lambda\in\mathbb{C}} \frac{r_\lambda}{x(y-\lambda)}+ \sum_{\lambda\in\mathbb{C}} \frac{s_\lambda}{x-\lambda} + g(x,y) $$
with $g$ smooth along $y=\lambda$ and $x=\lambda,\lambda\neq 0$.
Now the part $\tilde{f}_1$ has to exactly compensate these poles, but the derivatives for the $Ei,Li$ gives integrands which a linear combinations of terms of the form
$$\frac{x^{p-1}}{y-\lambda}, p\neq 0,\quad \frac{y^p}{x-\lambda},\; p\in\mathbb{N}^*,\; \lambda\neq 0$$
These can thus never compensate the poles of $\tilde{f}_2$. Thus $\tilde{f}_2$ had no such poles, nor $\tilde{f}_1$. This implies that $\tilde{f}_1=0$, and thus that $\tilde{f}$ has an elementary integral.

Now in step $5$, we know that $\tilde{f}$ should be elementary. Knowing that it has no multiple poles outside $x=0$, we have that its integral should write
$$P(x,y)+\sum \lambda_i \ln f_i(x,y) $$
with $P\in\mathbb{C}[x,\tfrac{1}{x},y]$ and $f_i$ factors of the denominator of $\tilde{f}$. Looking near the line at infinity, we have
$P(x,y)=x^\nu g_\nu(y/x)+x^{\nu-1} g_{\nu-1}(y/x) +\dots  $
and applying derivation gives
$$(\partial_x+\tfrac{1}{x}\partial_y ) (x^\nu g_\nu(y/x)+x^{\nu-1} g_{\nu-1}(y/x))=$$
$$x^{\nu-1}(\nu g_\nu-y/x g_\nu'(y/x))+x^{\nu-2}(\nu g_{\nu-1}(y/x)-y/x g_{\nu-1}'(y/x)+g_\nu'(y/x)-g_{\nu-1}(y/x))+\dots$$
Solving these differential equations, we find $g_\nu(y/x)=(y/x)^\nu$, but then for $g_{\nu-1}$, no rational solution is possible. Thus the derivation $\partial_x+\tfrac{1}{x}\partial_y $ can lower the degree of $P$ by at most $2$, which gives the bound of step $5$. The order of a pole at $x=0$ decreases exactly by one by differentiation, and thus the pole order at $x=0$ of $P$ should be one less than in $\tilde{f}$. Now the lasting part is the $K$ part, which is a linear combination of logs, and thus whose poles should appear as poles of $\tilde{f}$. Thus the $\ln K(x,\ln x)$ can be written as in step $5$, and so if no solution is found, then none exists.

\end{proof}

From this algorithm, the telescoper can be obtained by the formula
$$\partial_h^j\prod_{i\in  S} (\partial_h -i)$$
where $S$ is the set of exponents of $x$ appearing in the residues at $y=\lambda$ in $\tilde{f}$ and $j$ is the maximal degree of the polynomials in $y$ appearing in the residues at $x=\lambda$ in $\tilde{f}$.\\

\noindent
\textbf{Example} 1: $y=\ln x,\;\partial_x y= \frac{1}{x},\; U=1$\\
$$I_1=\int \frac{x^2}{(\ln x)^2} dx,\quad D_h I_1+ 3I_1=\frac{4x^3-y^2}{4y^2},\quad
I_1= -\frac{x^3}{\ln x} + 3Ei(3\ln x)$$
$$I_2=\int \frac{x^3+(\ln x)^3+x^2}{(x+1)\ln x} dx,\quad 
D_h^4 I_2 +3D_h^3 I_2= \frac{16y^4-162x^3}{27y^4}$$
$$I_2= 2\ln x Li_2(-x)-2Li_3(-x) + Ei(3\ln x)+ (\ln x)^2\ln(x+1)$$
$$I_3=\int \frac{2x(\ln x)^2+(\ln x)^3+(\ln x)^2-x-\ln x}{x(x+\ln x)(\ln x)^2},\quad
D_hI_3 = -\frac{4xy^2+4y^3-45y^2+45x+45y}{45y^2(y+x)},$$
$$I_3=\ln x + \frac{1}{\ln x} + \ln(x+\ln x)$$
For all of these cases, Maple heuristics manage to find the integral expression. However, if a mixture of elementary and exponential integrals is necessary, it fails
$$I_4= \int \frac{x^2+2x\ln x+\ln x}{x(x+\ln x) \ln x} dx,$$
$$D_h^2I_4+D_h I_4=\frac{14x^2y^2+28xy^3+14y^4-225x^3-450x^2y+225y^3-225y^2}{225y^2(y+x)^2}
,$$
$$I_4=Ei(\ln x) + \ln(x+\ln x)$$

\subsection{A Hamiltonian foliation}

We consider the Hamiltonian system given by the one degree of freedom Hamiltonian
$\mathcal{F}(x,y)=\ln x +\lambda\ln y + \mu \ln (x+y-1) $\\

Up to affine transformation and multiplication, any linear combination of logs of affine functions with no common zero can be written like this. Solving the system leads to the integral
$$I_{10}=\int_{\mathcal{F}(x,y)=h}  \frac{yx+y^2-y}{\lambda x+\lambda y+\mu y-\lambda} dx$$
In this example, we will manage to prove the non existence of a telescoper.

\begin{prop}
If $1,\lambda,\mu\in\bar{\mathbb{Q}}$ are $\mathbb{Q}$-independent and under Schanuel conjecture, the integral $I_{10}$ has no telescoper.
\end{prop}

\begin{proof}
If $I_{10}$ admitted a telescoper $L$, then it would have constant coefficients and we would satisfy the system
$$D_x I_{10}=\frac{yx+y^2-y}{\lambda x+\lambda y+\mu y-\lambda},\quad  LI_{10}=K$$
with $K\in\bar{\mathbb{Q}}(x,y)$. As the telescoper can be built from series with rational coefficients and linear system solving, the operator $L$ will have algebraic coefficients. We now consider factor $D_h-\alpha$ of $L$, with thus $\alpha\in\bar{\mathbb{Q}}$. By integration, we obtain an expression of the form
\begin{equation}\label{eqform}
\int e^{\alpha \mathcal{F}(x,y)}(P_1(x,y)dx+P_2(x,y)dy)   
\end{equation}
Now following the result of Theorem 1 of \cite{combot2019symbolic} (which requires Schanuel conjecture), we know that $e^{\alpha \mathcal{F}(x,y)}$ has to be decomposable, i.e. can be written under the form
$$e^{\alpha \mathcal{F}(x,y)}=f(V(x,y))W(x,y),\; V,W\in\mathbb{C}(x,y).$$
If $1,\alpha,\alpha\lambda,\alpha\mu$ are $\mathbb{Q}$-independant then the number of logs in $e^{\alpha \mathcal{F}(x,y)}$ cannot be reduced. Thus decomposition is impossible. Else there is a relation between $1,\alpha,\alpha\lambda,\alpha\mu$, and we can write
$$ e^{\alpha \mathcal{F}(x,y)}= A(x,y)^{1/k} F_1(x,y)^{\nu_1}F_2(x,y)^{\nu_2}$$
with $\nu_1,\nu_2$ $\mathbb{Q}$-independent. We can assume that $V$ is indecomposable (in the sense of composition of rational fractions \cite{cheze2014decomposition}) and then $V$ has to be the decomposition of each of the $F_i$. We know there are exactly two $F_i$ as $\alpha,\alpha\lambda,\alpha\mu$ are $\mathbb{Q}$-independent. We do not know the $F_i$, but as $\alpha \mathcal{F}(x,y)$ can be written as a linear combination of logs of the $F_i$ and $A$, the $F_i$ are product of powers of $x,y,x+y-1$. Thus $V$ is the decomposition of a rational function of the form
$$x^ny^m(x+y-1)^l$$
for at least two different non proportional triplets $(n,m,l),(n',m',l')$. Making the product of these expressions with suitable powers allows to eliminate the factor $x+y-1$, and we deduce that $V$ is the decomposition of a function of the form $x^ay^b$. Then restricting to the level $x^ay^b=1$, we should have
$$x^{n-am/b}(x+x^{-a/b}-1)^l$$
constant, which requires $l=0$, $(n,m)$ multiple of $(a,b)$. So it is impossible to have two non proportional triplets $(n,m,l),(n',m',l')$.

Thus for $\alpha\neq 0$, an integral of the form \eqref{eqform} should be exact, and then telescoper would reduce. Thus the telescoper has to be of the form $\partial_h^p$. Assume the minimal $p$ is $\geq 2$. Then the formula for $D_h^{p-2} I_{10}$ will contain an integral of a closed $1$ form $\omega=\omega_1 \mathcal{F}+\omega_0$ with $\omega_1,\omega_0$ with coefficients in $\mathbb{C}(x,y)$. We want to find for which forms $\omega_1$ the form $\omega$ is closed.

As $\omega$ is closed, the residues along the poles should be constant. Thus the poles of $\omega$ should be constant levels of $\mathcal{F}$, and as $1,\lambda,\mu\in\bar{\mathbb{Q}}$ are $\mathbb{Q}$-independent, only $x,y,x+y-1$ are possible. Because of the additive Galois action on $\mathcal{F}$, $\omega_1$ is closed, and thus
$$\int \omega_1=a\ln x +b \ln y + c\ln(x+y-1)+R(x,y)$$
with $R$ rational. Subtracting $\tfrac{c}{2\mu} d(\mathcal{F}^2)$ to $\omega$ removes the coefficient $c$, and so we can assume $c=0$. We now perform an integration by part, and we rewrite
$$\omega= d(\mathcal{F} (a\ln x +b \ln y)) -\left( d\mathcal{F} (a\ln x +b \ln y)+\tilde{\omega}_0\right)$$
with $\tilde{\omega}_0$ rational. Now along the pole $x+y-1=0$ of $d\mathcal{F}$, the residue cannot be constant because of the term in the numerator $a\ln x+b\ln(1-x)$, except if $a=b=0$.
Thus $\omega_1=0$ and $\omega$ is rational. Thus we had, before removing the coefficient $c$ by subtracting $\tfrac{c}{2\mu} d(\mathcal{F}^2)$,
$$\int \omega = \frac{c}{2\mu} \mathcal{F}^2+\int \tilde{\omega_0},$$
and in particular, $\omega_1$ should be a multiple of $d\mathcal{F}$. However, $\mathcal{F}^2$ is constant on the foliation, and thus this term can be removed as an integration constant. Thus $D_h^{k-2} I_{10}$ can be written as an integral of a rational $1$ form, and so has for telescoper $\partial_h$, and so $p$ was not minimal.

Now the only cases left are the telescopers $\partial_h,1$. In these cases, the integral is elementary, i.e. a linear combination of logs and a rational function. The singular curves of this integral are either poles of the integrand, or Darboux polynomials. Indeed, if a curve is transverse to the vector field, the Lie derivative will increase the order of the pole. As $1,\lambda,\mu\in\bar{\mathbb{Q}}$ are $\mathbb{Q}$-independent, the only Darboux polynomials are $x,y,x+y-1$, and thus we can write
$$I_{10}=a\ln x +b \ln y + c\ln(x+y-1)+R(x,y)$$
where $R$ is rational. We now need to control the order of the poles of $R$ which are Darboux polynomials. Let us consider a pole $P$ Darboux polynomial of multiplicity $n$. If $D_x$ decreases the order of a pole $P$, then 
$$D_x(R)=D_x\left(\frac{S}{P^n}\right)=O(P^{-(n-1)}).$$
In other words, $R$ is the beginning of a series expansion of a first integral which has a pole along $P=0$ of multiplicity $n$. Looking at our first integral, the possibilities respectively for $P=x,y,x+y-1$ are
$$e^{-n\mathcal{F}(x,y)},\; e^{-\frac{n}{\lambda}\mathcal{F}(x,y)},\; e^{-\frac{n}{\mu}\mathcal{F}(x,y)}$$
Even at first order, their expansion at $P=0$ have not rational functions coefficients, and thus $D_x$ cannot decrease the order of the poles $x,y,x+y-1$. Thus $R$ has for denominator at worst $xy(x+y-1)$. The pole at the line at infinity is similarly bounded, which gives
$$R=\frac{Q}{xy(x+y-1)},\quad \deg Q \leq 4$$
The problem reduces to a linear system solving which admits no solution.
\end{proof}

Remark that most of the proof did not require the expression of the integrand $I_{10}$. We proved in fact that any telescoper for an integral on this foliation should be either $\partial_h$ or $1$.

\section{Application to planar systems solving}

\subsection{Solvability of planar systems}

Consider a planar differential equation
\begin{equation}\label{eqAB}
\dot{x}=A(x,y),\; \dot{y}=B(x,y)    
\end{equation}
where $A,B\in\mathbb{Q}(x,y)$. Such system is typically solved in two steps
\begin{itemize}
    \item We first try to solve the non autonomous equation $\partial_x y(x)=(B/A)(x,y(x))$. Such equation is solvable by quadrature if and only if an integrating factor exists \cite{singer1992liouvillian}
    \item Compute an anti derivative of $1/A(x,y(x,h))$ in $x$ where $y(x,h)$ is a general solution of the previous equation. This integral is differentially algebraic if and only if it admits a telescoper.
\end{itemize}

\begin{defi}
We say that equation \eqref{eqAB} is solvable if equation $\partial_x y(x)=(B/A)(x,y(x))$ admits an integrating factor and $\int 1/A(x,y(x,h)) dx$ admits a telescoper with virtually solvable Galois group.
\end{defi}

Remark that under such conditions, equation \eqref{eqAB} is solvable by quadratures. For the reverse proposition, additional care should be added to the definition of quadrature: the dependence of the solutions with respect to the initial conditions should also be Liouvillian functions. This is not guaranteed as indeed, telescopers in the transcendental case do not always exist, and so it is possible that each solution of \eqref{eqAB} to be solvable by quadrature, but the general solution seen as a function of time and initial conditions to be differentially transcendental.

If equation $\partial_x y(x)=(B/A)(x,y(x))$ has no rational first integral, then $y(x,h)$ is transcendental in $x$, and all possible telescopers have constant coefficients, and thus Abelian Galois groups. In the algebraic case, telescopers can be more complicated, and we face the following dilemma.
$$I_1(x,h)=\int \frac{1}{\sqrt{x^3+h}} dx$$
For any given $h$, the function $I_1(x,h)$ is Liouvillian as it can be written by iterated integrals, exponentiations and algebraic extractions. As a two variable function, this is not enough as $I_1(x,h)$ has to be written with bivariate closed $1$ form integrals, and thus understanding on its dependence on $h$ is necessary. We find the $\partial$ finite PDE system
$$ h\partial_h I_1 +\tfrac{1}{6} I_1=\frac{1}{\sqrt{x^3+h}},\quad \partial_x I_1= \frac{1}{\sqrt{x^3+h}}$$
Now as this PDE system has a virtually solvable Galois group $\simeq \mathbb{Z}/6\mathbb{Z} \propto \mathbb{C}^{+}$, we can write
$$I(x,h)=h^{-1/6} \int \frac{3h^{1/6}dx-h^{-5/6}x dh }{3\sqrt{x^3+h}} $$
This unexpected sixth degree extension $h^{1/6}$ is invisible when looking in $x$ only, but this expression proves that $I_1$ is Liouvillian as bivariate function. The integral
$$I_2(x,h)=\int \frac{1}{\sqrt{x^3+x+h}} dx$$
however is Liouvillian in $x$, but not in $x,h$. This is because the modulus of the elliptic integral depends on $h$, and thus the telescoper will have for solutions complete elliptic integrals of the first kind. Those have non virtually solvable Galois groups, and thus $I_2$ will not be Liouvillian in $x,h$ (but is differentially algebraic). This problem is difficult as no general purpose efficient algorithm exists to compute the Galois group of a linear differential equation. For such equations coming from telescopers, we conjecture the following.

\begin{conj}
If a minimal order telescoper $T$ for an integral $\int a(x,h) dx$ with $a$ algebraic in $x,h$ has a virtually solvable Galois group, then it is finite.
\end{conj}

However, even if this conjecture holds, testing if the Galois group is finite is difficult as it can be very large, and thus this question in the algebraic case will be left apart.

\subsection{Integration of telescopers}

The purpose of this subsection is to present an implementation for Corollary \ref{cor1}.\\

\noindent\underline{\sf IntegrateTelescoper}\\
\textsf{Input:} Rational functions $F,G\in\mathbb{K}(x,y)$, a telescoper $T$ and optionally integrating factor $U$.\\
\textsf{Output:} A Liouvillian expression for $\int G(x,y(x,h)) dx$
\begin{enumerate}
\item If order $\hbox{ord}$ of $T$ is $0$, return the rational solution for $\int G(x,y(x,h)) dx$.
\item Note $M$ the companion matrix of the telescoper part of $T$, and compute $A$ its Jordan form with $P$ passage matrix.
\item Compute $B_2=UP^{-1}(0,\dots,0,C)$ where $C$ is the certificate part of $T$, and $B_1=P^{-1}((G,D_hG,\dots, D_h^{\hbox{ord}} G)-FB_2)$ where $D_h=U^{-1} \partial_y$.
\item Compute $E_1=\exp(-A\int -FUdx+Udy ), E_2=\exp(A\int -FUdx+Udy )$ and
$$V=\left(\int(E_1B_1)_idx+(E_1B_2)_idy\right)_{i=1\dots \hbox{ord}}$$
\item Return $(PE_2V)_1$
\end{enumerate}

\begin{prop}
Algorithm \underline{\sf IntegrateTelescoper} returns a Liouvillian expression in $x,y$ of the form \eqref{eqliouv} which, when evaluated on $y(x,h)$, is an integral of $G(x,y(x,h))$.
\end{prop}

\begin{proof}
We follow proof of Corollary \ref{cor1}. If order of $T$ is zero, then it can be written, $I(x,h)=R(x,y(x,h))$, meaning that the integral is rational, given by $R(x,y(x,h))$. So step $1$ is correct. Else we need to solve the PDE system
$$D_x I=G,\quad \sum\limits_{i=0}^{\hbox{ord}} a_i D_h^i I=C$$
where $D_h=U^{-1}\partial_y$ and $D_x=\partial_x+F\partial_y$. The solutions of the homogeneous part of this system are $(\exp(M\int -FUdx+Udy ))_1$ where $M$ is the companion matrix of the telescoper part of $T$. Step $2$ computes the Jordan form of $M$ and in step $3$ we perform the basis change on the non homogeneous part to put the whole system under Jordan form. Finally, we use variation of constant formula
$$e^{M\int -FUdx+Udy } \int e^{-M\int -FUdx+Udy } \left(U(0,\dots,0,C)dx +
((G,D_hG,\dots, D_h^{\hbox{ord}} G-FUC) dy \right)$$
to find the particular solution. In the new basis, the formula is just $E_2V$ where $V$ is computed in steps $3,4$. Then step $5$ returns the solution in the original basis, selecting the first line of it as it corresponds to $I$ (other lines are $D_h^i I$).
\end{proof}

Remark that the practical implementation in Maple of this algorithm relies on the capability of Maple to perform the integration in step $4$. It could seem that we did not progress, as the initial objective of the article was an integral computation, and that now we need to compute several ones! However, these integrals are always $\partial$ finite (in contrary to the initial integral which is typically not $\partial$ finite), thus are much easier to compute.

\subsection{Solver algorithm}

\begin{prop}\label{propds}
If equation \eqref{eqAB} is solvable and has no rational first integrals, then there exists $N\in\mathbb{N}$ such that algorithm \underline{PlaneDSolve} returns for almost all $(x_0,y_0)$ two relations
$$\{J_1(x,y)=c_1,\; J_2(x,y)=c_2+t\}$$
where $J_1,J_2$ are Liouvillian functions. If equation \eqref{eqAB} has a rational first integral $J_1$, then there exists $N\in\mathbb{N}$ such that algorithm \underline{PlaneDSolve} returns for almost all $(x_0,y_0)$ two relations
$$\left\lbrace J_1(x,y)=c_1,\; \int a(x,c_1) dx=c_2+t\right\rbrace$$
where $a$ is algebraic.
\end{prop}

As before, the point $(x_0,y_0)$ is used as the initial point for series expansions, and it should avoid a codimension $1$ algebraic set as else a FAIL answer could be returned. As we see, in the second case, the integral $\int a(x,c_1) dx$ always has a telescoper, but possibly not virtually solvable. And thus we cannot always remove the integration constant $c_1$ from inside the integral and write it as an integral of a closed $1$ form in $x,y$.\\

\noindent\underline{\sf PlaneDSolve}\\
\textsf{Input:} A vector field $(A,B)$ with $A,B\in\mathbb{Q}(x,y)$, a point $(x_0,y_0)$, an integer $N\in\mathbb{N}^*$.\\
\textsf{Output:} An implicit expression for the solutions, or FAIL, or ``None''
\begin{enumerate}
\item Look for rational first integrals of degree $2N$. If FAIL return FAIL. If one $J_1$, return
$$\left\lbrace J_1(x,y)=c_1,\; \int a(x,c_1) dx=c_2+t\right\rbrace$$
where $a(x,c_1)$ is $1/A(x,y(x,c_1))$ with $y(x,c_1)$ a solution of $J_1(x,y)=c_1$.
\item Search for an integrating factor of degree $\leq N$. If FAIL return FAIL. If ``None'' return ``None''. Else note it $U$.
\item For all $\deg,\hbox{ord}\in\mathbb{N}^2$ such that $\tfrac{1}{2}(\deg+1)(\deg+2)(\hbox{ord}+2)\leq \tfrac{3}{2}(N+1)(N+2)$
search a telescoper for $1/A$ with equation $\partial_x y= (B/A)(x,y)$, point $(x_0,y_0)$, degree $\deg$, order $\hbox{ord}$ and integrating factor $U$. If FAIL return FAIL. If none, return ``None''.
\item Else note one $T$, and apply \underline{\sf IntegrateTelescoper} with it, giving a Liouvillian expression $J_2$. Return
$$\left\lbrace \int -\tfrac{B}{A}Udx+Udy=c_1,\; J_2(x,y)=c_2+t\right\rbrace$$
\end{enumerate}

\begin{proof}
By construction, the algorithm terminates, and returns an answer of the possible forms of Proposition \ref{propds}. If $(A,B)$ admits a rational first integral, then it will be discovered for $N$ large enough in step $1$, and thus step $1$ will return the solution, or FAIL if $(x_0,y_0)$ belongs to a codimension $1$ algebraic set (so of zero measure).
If $(A,B)$ does not admit a rational first integral but an integrating factor, then step $2$ will discover it for $N$ large enough. Then in step $3$ we search for telescopers. For any pair $(\deg,\hbox{ord})$, a telescoper with these degree and order will be found if $N$ is large enough. Then step $4$ will integrate it (possibly only in a formal way with keeping the $\int$ symbol), and return the Liouvillian first integral $\int -\tfrac{B}{A}Udx+Udy$ and the expression for the integral of $1/A$ along the foliation given by equation $\partial_x y= (B/A)(x,y)$.
\end{proof}

Remark that obtaining the FAIL answer is a good sign. Indeed, either $(x_0,y_0)$ was very badly chosen (zero probability), or the system admits a rational first integral which provoked a FAIL answer in step $2$ or $3$. The bound $2N$ in step $1$ is chosen for practical reasons such that time computation of this step is similar to the next ones. We could ensure that for any given $N$, the FAIL answer only occurs on a codimension $1$ algebraic set, but this would require to check the existence of rational first integral of degree $O(N\hbox{ord}^3)$ which would cost a lot computation time for typically nothing.

Applying \underline{\sf PlaneDSolve} for all $N\in\mathbb{N}$ gives a semi algorithm for the solving \eqref{eqAB} by quadrature: if it is indeed solvable, the solution will be found. However, if it is not, then \underline{\sf PlaneDSolve} will never be able to prove it. It is nevertheless better than heuristics available, as it can never miss a solution, provided that $N$ is large enough, which encodes somehow the complexity of the solution the user is searching for.

\section{Examples}

\subsection{Simple extensions}

We can take for $y(x)$ an elementary expression, which happen to satisfy a differential equation, and use the telescoper algorithm to find expressions for integrals. We will restrict ourselves to $\mathbb{K}=\mathbb{Q}$, which is an actual restriction for the implementation.\\

\noindent
\textbf{Example} 2: $y=\left(\frac{x-\sqrt{2}}{x+\sqrt{2}}\right)^{\sqrt{2}},\; \partial_x y=\frac{4y}{x^2-2},\; U=\frac{1}{y}$\\
$$I_5=\int \left(\frac{x-\sqrt{2}}{x+\sqrt{2}}\right)^{\sqrt{2}} dx,\quad
D_h I_5-I_5=0,\quad
I_5=\int \left(\frac{x-\sqrt{2}}{x+\sqrt{2}}\right)^{\sqrt{2}} dx  $$
This example $I_5$ seems ridiculous as no integration has been performed effectively. The integrals in the answer are guaranteed to be integrals of hyperexponential forms in $x,y$, and in this example, the initial integrand in $x$ was already hyperexponential.
$$I_6=\int \frac{(x^2+2) \left(\frac{x-\sqrt{2}}{x+\sqrt{2}}\right)^{\sqrt{2}} }{(x^2-2)^2\left(\left(\frac{x-\sqrt{2}}{x+\sqrt{2}}\right)^{\sqrt{2}} + 1\right)} dx,\; D_h^2I_6-\tfrac{1}{2}I_6= \frac{2x^2y^2+5x^2y+2xy^2+2x^2+2xy-4y^2-6y-4}{4(y+1)^2(x^2-2)}$$
$$I_6= -\frac{(6x^5+40x^3+24x)\sqrt{2}+x^6+30x^4+60x^2+8}
{(32x^5-128x)\sqrt{2}+8x^6+80x^4-160x^2-64}\Phi\left(-\left(\frac{x-\sqrt{2}}{x+\sqrt{2}}\right)^{\sqrt{2}},1,-\frac{\sqrt{2}}{2}\right)$$
$$-\frac{(2x\sqrt{2}+x^2+2)(x^2-2)}{(32x^3+64x)\sqrt{2}+8x^4+96x^2+32}\Phi\left(-\left(\frac{x-\sqrt{2}}{x+\sqrt{2}}\right)^{\sqrt{2}},1,\frac{\sqrt{2}}{2}\right)
-\frac{(x+2)(x-1)}{x^2 - 2}$$
where $\Phi$ is the Lerch function.

$$I_7=\int \frac{(x^2+2x+2)\left(\frac{x-\sqrt{2}}{x+\sqrt{2}}\right)^{2\sqrt{2}}+(x^2-2)\left(\frac{x-\sqrt{2}}{x+\sqrt{2}}\right)^{\sqrt{2}}}
{(x^2-2)\left((x^2-2)\left(\frac{x-\sqrt{2}}{x+\sqrt{2}}\right)^{2\sqrt{2}}+2(x^2+2)\left(\frac{x-\sqrt{2}}{x+\sqrt{2}}\right)^{\sqrt{2}}+x^2-2\right)} dx$$
$$ D_h^2I_7+2D_hI_7-I_7=$$
\begin{small} $$\frac{((7y^2+14y+3)x+16y^2)(y+1)^2x^3-4(3y^4-16y^3-6y^2+8y+3)x^2-4(8y^2x+y^2-2y-3)(y-1)^2}{16(x^2y^2+2x^2y+x^2-2y^2+4y-2)^2} 
$$ \end{small}
In this last case, Maple is not able to perform the computation of integrals required by algorithm \underline{\sf IntegrateTelescoper} correctly. The integrals are integrals of hyperexponential $1$-forms, and thus can always be represented using single variable integrals \cite{combot2019symbolic}, giving the following formula
$$
I_7=e^{\frac{1}{1+\sqrt{2}}(\ln y  + \sqrt{2}\ln\left(\frac{x+\sqrt{2}}{x -\sqrt{2}} \right)}
\int^{y\frac{x-\sqrt{2}}{x+\sqrt{2}}} \frac{z^{-\frac{1}{1 + \sqrt{2}}}}{8(z + 1)} dz + e^{\frac{1}{1-\sqrt{2}}(\ln y  + \sqrt{2}\ln\left(\frac{x+\sqrt{2}}{x -\sqrt{2}} \right)}
\int^{y\frac{x+\sqrt{2}}{x -\sqrt{2}}} \frac{z^{-\frac{1}{1 - \sqrt{2}}}}{8(z + 1)} dz
$$
Remark that similarly to Proposition \ref{propex}, all these integrals can be expressed using the Lerch function $\Phi$.\\

\noindent
\textbf{Example} 3: $y=e^{-x^2}\sqrt{\pi}\hbox{erf}(x)$\\
$$I_8=\int \frac{e^{3x^2}\hbox{erf}(x)^3 \pi^{3/2} + 3\pi \hbox{erf}(x)^2e^{2x^2} + 2e^{x^2}\sqrt{\pi} \hbox{erf}(x)x^2+2x}{e^{x^2} \hbox{erf}(x)\sqrt{\pi}x} dx$$
$$D_h^3 I_8=-\frac{2e^{3x^2}}{y^3},\quad I_8=\int \frac{\hbox{erf}(x)^2 e^{2x^2}\pi +3\sqrt{\pi}\hbox{erf}(x) e^{x^2} }{x} dx +  \ln\left( e^{x^2}\hbox{erf}(x)\right)$$
The remaining integral cannot be written using a classical special function. However, this last expression is significantly better than the previous one as the integral is now a D finite integral. As given by Corollary \ref{cor1}, all resulting integrals are D finite (even Liouvillian), contrary to an input in $\mathbb{C}(x,e^{-x^2}\hbox{erf}(x))$ which is typically not D finite.\\

\noindent
\textbf{Example} 4: $y=\frac{\hbox{tan}(\Pi(x,-1,2))}{\sqrt{(x^2-1)(4x^2-1)}}$, 
$\partial_x y= \frac{1+(4x^4-5x^2+1)y^2-(8x^5+3x^3-5x)y}{(4x^2-1)(x^4-1)}$, 
$U=-\frac{\sqrt{(x^2+1)(4x^2-1)}}{1+(4x^4-5x^2+1)y^2}$\\
$$I_9=\int \frac{\hbox{tan}(\Pi(x,-1,2))^4 +2(1-\hbox{tan}(\Pi(x,-1,2))^2)\frac{\hbox{tan}(\Pi(x,-1,2))}{\sqrt{(x^2-1)(4x^2-1)}} + 2\hbox{tan}(\Pi(x,-1,2))^2}
{2(4x^6 - x^4 - 4x^2 + 1)(1 + \hbox{tan}(\Pi(x,-1,2))^2) \frac{\hbox{tan}(\Pi(x,-1,2))}{\sqrt{(x^2-1)(4x^2-1)}}} dx$$
$$D_h^3I_8+4D_h I_8=-\frac{(4x^4y^2-5x^2y^2+y^2+1)^3}{((x^2-1)(4x^2-1))^{3/2}y^3}$$
The integral $I_8$ can be written with integrals of $e^{-2i\Pi(x,-1,2)}$ and elementary functions.\\

The cases with an algebraic integrating factor are the most complicated, as they typically involve superelliptic integrals. Sadly, Weierstrass elliptic functions do not satisfy a first order rational differential equation (it is algebraic), and thus do not fit in this framework.

\subsection{Darbouxian foliations}

\noindent
\textbf{Example} 5: $\mathcal{F}(x,y)=\ln (1+x^2y)+
\sqrt{2}\ln\left(\frac{x^2+y\sqrt{2}}{x^2-y\sqrt{2}}\right)$\\
This is a Darbouxian foliation with a Darboux polynomial $x$ which is not a singularity of $\mathcal{F}$ but even more on which $d\mathcal{F}$ vanishes.\\
$$I_9=\int_{\mathcal{F}(x,y)=h} \!\!\!\!\!\!\!\! \frac{y^2x(x^2y^6-4y^7-4x^2y^4+19y^5+x^2y^2+2y^3+2x^2-8y)}{(x^4+4x^2y-2y^2+4)(y^2+2)^5} dx$$
We find a telescoper of order $0$, and thus a rational expression for $I_9=$
$$\frac{10x^4y^8+242x^4y^6-81x^2y^7+402x^4y^4-162x^2y^5+590x^4y^2-81y^6+376x^4-324y^4-324y^2}{648(y^2+2)^4x^4}$$
Remark that the integrand has no pole at $x=0$, but the integral expression has a pole at $x=0$ with multiplicity $4$. In contrary to rational integration, the integral even if rational can display new poles. Of course, if one would substitute $y$ by its expression in $x$, we could directly see if the integrand and the integral has indeed a pole as a function in $x$. However, expression of $y(x)$ will be here very complicated.\\

\subsection{Vector fields}

\noindent
\textbf{Example} 7: $\mathcal{F}(x,y)=\ln x +\lambda\ln y + \mu \ln (x+y) $
$$ \dot{x}=\frac{(\mu+\lambda)y+\lambda x}{y(y+x)},\; \dot{y}=-\frac{\mu x+x+y}{x(y+x)}$$
This is a Hamiltonian vector field, and thus automatically possesses a first integral, here $\mathcal{F}(x,y)$. The corresponding integrating factor is
$$U=\frac{\lambda x+\lambda y+\mu y}{\lambda\mu y(x+y)}$$
Following the previous proof of section 4, the reasoning fails as a pullback function $V$ is now possible, $V=x/y$. In fact the first integral can be rewritten
$$\mathcal{F}(x,y)=(1-\lambda-\mu)\ln x +\lambda\ln (y/x) + \mu \ln (1+y/x) $$
and thus $\exp(\mathcal{F}(x,y)/(1-\lambda-\mu))$ is now decomposable. Thus non trivial telescopers are a priori possible. Integration in time for $x$ requires the computation of the integral
$$I_{11}=\int_{\mathcal{F}(x,y)=h}  \frac{y(x+y)}{\lambda x+(\lambda +\mu) y} dx$$
As the current implementation only supports rational coefficients, we have to choose for $\lambda,\mu$ rational numbers thus producing a rational first integral (in principle we could do the computations in $\mathbb{K}=\mathbb{Q}(\lambda,\mu)$, but current implementation does not terminate). However, if large enough, the rational first integral will have no consequences on the calculations. For $\lambda=1/3,\; \mu=1/5$, we find the telescoper for the integral
$$D_hI_{11}-\frac{2}{23} I_{11}=-\frac{15(13x^2y+13xy^2-10x-16y}{299(5x+8y)}$$
and the solution of the differential system
$$-\tfrac{30}{13}\ln x- \tfrac{10}{13} \ln y - \tfrac{6}{13} \ln(x + y) = c_1,$$
$$x^{\frac{30}{23}} y^{\frac{10}{23}} (x+y)^{\frac{6}{23}}
\int \frac{\tfrac{15}{x}(13x^2y+13xy^2+36x+30y)dx-\tfrac{15}{y}(13x^2y+13xy^2-10x-16y)dy}{299x^{30/23}y^{10/23}(x+y)^{29/23}}=t+c_2$$
Such result can be generalized by guessing the telescoper for several $\lambda,\mu$ and checking it, giving the telescoper formula
$$D_hI_{11} - \frac{2\mu\lambda}{\mu+1+\lambda} I_{11} =
-\frac{\mu\lambda (2\lambda x^2y+2\lambda xy^2+\mu x^2y+\mu xy^2-2\lambda x - 2\lambda y - 2\mu y)}{(2\lambda+\mu)(\mu + 1 +\lambda)(\lambda x + \lambda y +\mu y)}$$
and the solution of the differential system
$$\ln x  + \lambda \ln y + \mu \ln(x+y)=c_1, t+c_2=e^{2\frac{\ln x + \lambda\ln y + \mu\ln(x+y)}{\mu+1+\lambda}} \times$$
$$ \int
\frac{\tfrac{1}{x}((2\lambda+\mu)(x^2y+xy^2) +2(\mu+1)x+2y) dx -\tfrac{1}{y}(2(x+y)(xy-1)\lambda + y\mu (x^2+xy-2))dy}
{e^{2\frac{\ln x + \lambda\ln y + \mu\ln(x+y)}{\mu+1+\lambda}}(2\lambda+\mu)(\mu+1+\lambda)(x+y)} $$
As proved in \cite{combot2019symbolic}, this last integral can then be expressed using a single variable integral which can be written in terms of a (solvable) hypergeometric function.\\

\noindent
\textbf{Example} 8:
$$\dot{x}=-\frac{xy(y^2-2)}{xy^2-2x+4y},\;\;\dot{y}=-\frac{y(y^2-2)^2}{4xy^2-8x+16y}$$
The solution returned is
$$\tfrac{1}{4}\ln(x) - \tfrac{\sqrt{2}}{2}\hbox{arctanh}\left(y\tfrac{\sqrt{2}}{2}\right)  = c_1,\quad 
x e^{-2\sqrt{2}\hbox{arctanh}\left(\tfrac{\sqrt{2}}{2}y\right)}
\int e^{2\sqrt{2} \hbox{arctanh}\left(\tfrac{\sqrt{2}}{2}y\right)} y^{-2} dy +$$
$$\tfrac{21\sqrt{2}}{2}\hbox{arctanh}\left(\tfrac{\sqrt{2}}{2}y\right) - \tfrac{25}{4}\ln x+
\frac{4xy^2-25y^3+16y^2-8x+50y}{4y(y^2-2)} = t + c_2$$
We recognize the foliation of example 2, and the unevaluated integral could be expressed in terms of the Lerch function $\Phi$. Maple dsolve command does not seem to terminate on this example.\\

\noindent
\textbf{Example} 9:
$$\dot{x}=y,\;\; \dot{y}=-\frac{4x^3-3xy^2-4x^2 + y^2}{2(x-1)x}$$
The solution returned is
$$-\frac{x(x-1)^2}{4x^2-y^2-4x}=c_1,\quad  - \mathfrak{F}\left(\frac{y \sqrt{x-1}}{\sqrt{4x^3-4x^2-y^2}}, \frac{\sqrt{4x^3-4x^2-y^2}}{\sqrt{x}(2x-2)}\right) =t+c_2$$
where $\mathfrak{F}$ denotes the elliptic integral of the first kind. This vector field admits a rational first integral, and the time integration involves an elliptic integral whose modulus depends on the value of the rational first integral. Thus the telescoper for the integral will be equivalent to the equation for the complete integral of the first kind, and so not solvable. In particular, the integral represented with $\mathfrak{F}$ is not Liouvillian as a two variables function in $x,y$.

\section{Conclusion}

If an integral has a telescoper, then for a generic initial point $(x_0,y_0)$, the algorithms will find one provided that the degree and order bounds are large enough. However, nothing guarantees that the telescoper order will be minimal. If it is not, then during the integration process, there will be cancellation between integrals. To detect such cancellation, we would need an equivalent of Hermite reduction for the differential forms. In the case of hyperexponential forms, we have a reduction algorithm \cite{combot2019symbolic}. It seems that similarly for more complicated foliations, the integrals can always be represented as composition of single variable integrals and algebraic functions in $x,y$. Such representation would allow simplification of the formulas and possibly represent the solutions with standard special functions. For some foliations as given in section 4, more specific algorithms are probably possible once the possible expressions for integrals admitting a telescoper are known. Finally, for the problem of solving planar differential systems, a last step would be the inversion, so to write $x,y$ in function of $t$. This inversion is typically very complicated, but looking for meromorphic solutions in time, it seems that only functions involving $\exp,\wp$ and algebraic extensions are necessary.

\label{}
\bibliographystyle{plain}
\bibliography{bibthese}

\end{document}